\documentclass[11pt,reqno]{amsart}

\usepackage[T1]{fontenc}
\usepackage[utf8]{inputenc}
\usepackage{lmodern}
\usepackage{amsmath,amssymb,amsfonts,amsthm,mathtools}
\usepackage{geometry}
\usepackage{microtype}
\usepackage{hyperref}
\usepackage{aliascnt}
\usepackage[nameinlink,noabbrev]{cleveref}

\hypersetup{%
  colorlinks=true,
  linkcolor=blue,
  citecolor=blue,
  urlcolor=blue,
  pdftitle={A Weighted Bossel--Daners Transfer Principle and a Pure-Power Robin Faber--Krahn Inequality},
  pdfauthor={Tan Duc Do, Nguyen Ngoc Trong, Nguyen Ngoc Huy Truong},
  pdfsubject={Weighted Robin eigenvalue comparison and pure-power Faber--Krahn inequality},
  pdfkeywords={Robin eigenvalue, Faber--Krahn inequality, Bossel--Daners method, weighted perimeter, power weights}
}
\allowdisplaybreaks
\numberwithin{equation}{section}

\newtheorem{theorem}{Theorem}[section]
\newaliascnt{lemma}{theorem}
\newtheorem{lemma}[lemma]{Lemma}
\aliascntresetthe{lemma}
\newaliascnt{proposition}{theorem}
\newtheorem{proposition}[proposition]{Proposition}
\aliascntresetthe{proposition}
\newaliascnt{corollary}{theorem}
\newtheorem{corollary}[corollary]{Corollary}
\aliascntresetthe{corollary}
\theoremstyle{definition}
\newaliascnt{definition}{theorem}
\newtheorem{definition}[definition]{Definition}
\aliascntresetthe{definition}
\newaliascnt{assumption}{theorem}
\newtheorem{assumption}[assumption]{Assumption}
\aliascntresetthe{assumption}
\theoremstyle{remark}
\newaliascnt{remark}{theorem}
\newtheorem{remark}[remark]{Remark}
\aliascntresetthe{remark}

\crefname{theorem}{theorem}{theorems}
\Crefname{theorem}{Theorem}{Theorems}
\crefname{lemma}{lemma}{lemmas}
\Crefname{lemma}{Lemma}{Lemmas}
\crefname{proposition}{proposition}{propositions}
\Crefname{proposition}{Proposition}{Propositions}
\crefname{corollary}{corollary}{corollaries}
\Crefname{corollary}{Corollary}{Corollaries}
\crefname{definition}{definition}{definitions}
\Crefname{definition}{Definition}{Definitions}
\crefname{assumption}{assumption}{assumptions}
\Crefname{assumption}{Assumption}{Assumptions}
\crefname{remark}{remark}{remarks}
\Crefname{remark}{Remark}{Remarks}

\newcommand{\R}{\mathbb R}
\newcommand{\Haus}{\mathcal H}

\title[A weighted Robin Bossel--Daners transfer]{A Weighted Bossel--Daners Transfer Principle and a Pure-Power Robin Faber--Krahn Inequality}
\author[T. D. Do]{Tan Duc Do}
\address{$^{1}$Faculty of Applied Sciences, HCMC University of Industry and Trade\\
140 Le Trong Tan Street, Tan Phu District, Ho Chi Minh City, Vietnam.}
\email{\textcolor[rgb]{0.00,0.00,0.84}{tanducdo.math@gmail.com}}

\author[N. N. Trong]{Nguyen Ngoc Trong}
\address{$^{2}$Group of Analysis and Applied Mathematics, Department of Mathematics, Ho Chi Minh City University of Education, Vietnam.}
\email{\textcolor[rgb]{0.00,0.00,0.84}{trongnn@hcmue.edu.vn}}

\author[N. N. H. Truong]{Nguyen Ngoc Huy Truong$^*$}\thanks{$^*$Corresponding author}
\address{$^{3}$Ho Chi Minh City University of Education, Vietnam.}
\email{\textcolor[rgb]{0.00,0.00,0.84}{hitruongofficial@gmail.com}}
\date{}

\subjclass[2020]{35P15, 35J92, 35J70, 49Q20}
\keywords{Robin eigenvalue, Faber--Krahn inequality, Bossel--Daners method, weighted perimeter, weighted isoperimetry, power weights, double-density isoperimetry}

\begin{document}

\begin{abstract}
We prove a weighted Bossel--Daners transfer principle for the first Robin eigenvalue of the $p$-Laplacian under $w=m^{1/p'}$, where $p'=p/(p-1)$. The argument combines double-density isoperimetry with spectral admissibility for singular weights and normalized-flux monotonicity. It uses an exact $BV$ zero-extension formula, an $L^{p'}(m\,dx)$ selection lemma, and a nonatomic rank map that remains well defined on positive-measure level sets.

For the singular pair
\[
m_b(x)=|x|^b,
\qquad
w_b(x)=|x|^{b/p'},
\]
known power-weight isoperimetry provides the geometric input. We establish spectral admissibility up to the critical exponent $p=N$ and, for the positive radial first eigenfunction $z=z(r)$ on the centered ball $B_R$, derive the integrated singular radial equation and center asymptotics and prove directly that
\[
\theta_R'(r)>0
\quad(0<r<R),
\qquad
\theta_R(r)=\frac{|z'(r)|^{p-1}}{r^{b/p'}z(r)^{p-1}},
\]
with $\theta_R(0)=0$ and $\theta_R(R)=\beta$. Let $\Omega\subset\R^N$ be a finite union of bounded connected Lipschitz domains whose closures are pairwise disjoint, let $\Omega_b^\sharp$ be the centered ball of equal $|x|^b$-weighted volume, and let $\lambda_{1,\beta}^b$ denote the first eigenvalue for the displayed pair. Consequently,
\[
\lambda_{1,\beta}^{b}(\Omega_b^\sharp)
\le
\lambda_{1,\beta}^{b}(\Omega)
\]
whenever $N\ge2$, $1<p\le N$, $-p<b<0$, and $\beta>0$. The range $1<p<N$ is complementary to the previously known $p\ge N$ weighted Talenti theory; at the shared endpoint $p=N$, the present proof also permits the singular contact $0\in\partial\Omega$.
\end{abstract}

\maketitle

\section{Introduction}

Let $N\ge2$, $1<p<\infty$, $\beta>0$, and let $\Omega\subset\R^N$ be a bounded Lipschitz open set. Throughout, a domain is connected, while a bounded Lipschitz open set may be disconnected and is understood to be a finite union of bounded connected Lipschitz domains whose closures are pairwise disjoint. The first Robin eigenvalue of the $p$-Laplacian is
\[
\lambda_{1,\beta}(\Omega)
:=
\inf_{0\ne u\in W^{1,p}(\Omega)}
\frac{\displaystyle\int_\Omega|\nabla u|^p\,dx
+\beta\int_{\partial\Omega}|\operatorname{Tr}u|^p\,d\Haus^{N-1}}
{\displaystyle\int_\Omega|u|^p\,dx}.
\]
The Robin Faber--Krahn inequality states that balls minimize this eigenvalue under a volume constraint. Bossel proved the planar linear result, and Daners established the inequality in arbitrary dimension \cite{Bossel,Daners}. For the nonlinear $p$-Laplacian, Bucur and Daners introduced a flexible weak level-set proof, while Dai and Fu obtained the comparison independently by a different route \cite{BucurDaners,DaiFu}. The level-set method is especially useful for Robin problems because it avoids a Robin P\'olya--Szeg\H{o} inequality for arbitrary Sobolev functions.

Weighted Robin spectral inequalities are substantially more delicate because the weights governing diffusion, volume, and boundary interaction need not coincide. Chiacchio and Gavitone treated the Hermite operator with Robin boundary conditions \cite{ChiacchioGavitone}. Amato, Chiacchio, and Gentile studied weighted $p$-Poisson and Robin eigenvalue comparisons with variable boundary coefficients \cite{ACG}. Their eigenvalue theorem assumes $p\ge N$, $-N<b<0$, and $0\notin\partial\Omega$; it is therefore complementary to the range proved here when $1<p<N$ and overlaps it only at $p=N$, subject to their additional exclusion of boundary contact with the origin. Chen, Li, and Wei recently developed the Bossel--Daners method for the Robin $p$-Laplacian on complete Riemannian manifolds \cite{ChenLiWei}. Chen, Wang, and Zhu proved weighted linear Robin comparisons for smooth radially log-convex densities under additional conditions on the density or the Robin parameter \cite{ChenWangZhu}; in their variational structure the same density weights diffusion, volume, and boundary interaction.

The representation--selection architecture itself is not new: in the linear same-density setting, Propositions~3.1--3.3 of \cite{ChenWangZhu} provide the superlevel representation, differential comparison, and level selection. Here it is adapted to a nonlinear compatible pair with unweighted diffusion, $BV$ traces, an $L^{p'}(m\,dx)$ selection argument, and a rank construction that is insensitive to plateaus. No unique-continuation or regularity result available under our rough-weight hypotheses is used to exclude positive-measure level sets, and we make no assertion that such plateaus actually occur.

The present paper addresses a different weighted structure. Let
\[
d\mu=m(x)\,dx
\]
be a radial weighted volume measure, put $p'=p/(p-1)$, and set
\[
w=m^{1/p'}.
\]
Thus this is one-density data in a compatible form: $m$ determines both the volume density and, through $m^{1/p'}$, the boundary density, while diffusion remains unweighted. For $m_b(x)=|x|^b$ with $b<0$, the interior operator is of Hardy type.
For a bounded Lipschitz domain on which the integrals are finite, define
\[
\lambda_{1,\beta}^{m,w}(\Omega)
:=
\inf_{\substack{u\in W^{1,p}(\Omega)\\ \int_\Omega m|u|^p\,dx>0}}
\frac{\displaystyle\int_\Omega|\nabla u|^p\,dx
+\beta\int_{\partial\Omega}w|\operatorname{Tr}u|^p\,d\Haus^{N-1}}
{\displaystyle\int_\Omega m|u|^p\,dx}.
\]
The relation $w=m^{1/p'}$ is dictated by the level-set H\"older estimate
\begin{equation}
\left(\int_{\{u=t\}}m^{1/p'}\,d\Haus^{N-1}\right)^p
\le
\left(\int_{\{u=t\}}|\nabla u|^{p-1}\,d\Haus^{N-1}\right)
\left(\int_{\{u=t\}}\frac{m}{|\nabla u|}\,d\Haus^{N-1}\right)^{p-1}.
\label{eq:intro-coarea-holder}
\end{equation}
Thus the associated weighted perimeter of a finite-perimeter set $E$ is
\[
P_w(E)=\int_{\partial^*E}w\,d\Haus^{N-1},
\]
where $\partial^*E$ denotes the reduced boundary. If $E^\sharp$ is the centered ball satisfying $\mu(E^\sharp)=\mu(E)$, the geometric input is
\[
P_w(E)\ge P_w(E^\sharp).
\]

\subsection{The weighted transfer}
Our central contribution is an abstract transfer theorem based on three inputs:
\[
\boxed{\text{compatible isoperimetry}
+\text{spectral admissibility}
+\text{regular monotone radial flux}.}
\]
On the centered comparison ball $B_R$, let $z=z(r)$ be the positive radial first eigenfunction. Writing $\mathfrak w(r)$ for the radial profile of $w$, set
\[
\theta_R(r)=\frac{|z'(r)|^{p-1}}{\mathfrak w(r)z(r)^{p-1}}.
\]
Besides continuity and monotonicity of $\theta_R$, the ball identity requires an integrated radial equation and a relative center-decay condition for the radial flux; the endpoint values $\theta_R(0)=0$ and $\theta_R(R)=\beta$ are then consequences. If the three displayed inputs hold together with the natural finite boundary-integrability conditions, we prove
\[
\lambda_{1,\beta}^{m,w}(\Omega^\sharp)
\le
\lambda_{1,\beta}^{m,w}(\Omega).
\]
The proof contains three points that are not formal consequences of the unweighted argument. First, the weighted selection lemma is established under the natural condition $\varphi\in L^{p'}(\Omega,d\mu)$. Second, the superlevel perimeter must include the boundary trace exactly:
\[
P_w(\{u>t\})
=
\int_{\partial^*\{u>t\}\cap\Omega}w\,d\Haus^{N-1}
+
\int_{\{\operatorname{Tr}u>t\}\cap\partial\Omega}w\,d\Haus^{N-1}
\]
for almost every $t$, where $\operatorname{Tr}u$ is the Sobolev trace. Third, the usual distribution-function parametrization is ambiguous on positive-measure level sets. We replace it with a nonatomic rank map that treats every plateau, including a possible maximum plateau.

When $m=w=1$, the Euclidean isoperimetric inequality and the radial equation verify the assumptions. The connected-domain comparison, the component identity
\[
\lambda_{1,\beta}(\Omega)=\min_j\lambda_{1,\beta}(\Omega_j)
\]
and strict decrease of $R\mapsto\lambda_{1,\beta}(B_R)$ give a consistency check with the Bucur--Daners inequality under our standing finite-component convention. Stronger formulations for finite-perimeter and free-discontinuity classes are already known \cite{BucurGiacomini2010,BucurGiacomini2015,BucurFreitasKennedy} and are not recovered by the present Lipschitz transfer.

\subsection{The pure-power theorem and its difficulties}
Our concrete application is the singular compatible pair
\begin{equation*}
m_b(x)=|x|^b,\qquad w_b(x)=|x|^{b/p'}.
\end{equation*}
The compatible geometric inequality is
\begin{equation*}
\int_{\partial^*E}|x|^{b/p'}\,d\Haus^{N-1}
\ge
\int_{\partial E_b^\sharp}|x|^{b/p'}\,d\Haus^{N-1},
\end{equation*}
where $E_b^\sharp$ is the centered ball having the same $|x|^b$-weighted volume as $E$. This geometric inequality is a specialization of the power-weight theory of Alvino, Brock, Chiacchio, Mercaldo, and Posteraro \cite{ABCMP}; no geometric novelty is claimed here.

The radial map
\[
T(x)=|x|^{\gamma-1}x,\qquad \gamma=\frac{N+b}{N},
\]
has Jacobian determinant $\det DT(x)=\gamma|x|^b$ and reduces the compatible perimeter to the increasing power perimeter with exponent
\[
a_{N,p,b}=\frac{b(p-N)}{p(N+b)}.
\]
For $1<p\le N$ and $-N<b<0$, one has $a_{N,p,b}\ge0$. The power-weight theorem of \cite{ABCMP} covers every such nonnegative exponent; the narrower condition $b>-p$ is imposed by spectral compactness and the radial center scaling, not by the geometry.

If $b>0$ and $p<N$, the same transformation gives $a_{N,p,b}<0$, outside the centered increasing-power inequality used here. This is consistent with the possibility of symmetry breaking in related H\'enon-type problems; see \cite{SmetsSuWillem,Serra}.

The analytic verification has two separate singularities. The volume density may diverge at the origin, while the boundary density may diverge at a contact point $0\in\partial\Omega$. We prove compactness, boundedness, positivity, and simplicity of the first eigenfunction for the full range $1<p\le N$ and $-p<b<0$. On the centered ball, the radial eigenfunction satisfies
\[
r^{N-1}|z'(r)|^{p-1}
=
\lambda_R\int_0^r s^{N-1+b}z(s)^{p-1}\,ds.
\]
The absence of an integration constant here is exactly where $p\le N$ enters the radial argument. The identity gives the center asymptotic
\[
\theta_R(r)=\frac{\lambda_R}{N+b}r^{1+b/p}+o(r^{1+b/p}).
\]
Moreover,
\[
\theta_R'(r)
=
r^{b/p}\bigl[\lambda_R+(p-1)\theta_R(r)^{p'}\bigr]
-
\left(N-1+\frac b{p'}\right)\frac{\theta_R(r)}r.
\]
After the change of variable $t=r^{1+b/p}/(1+b/p)$, every positive critical point of the transformed flux would be a strict local minimum. Since the derivative is positive near the origin, a first zero cannot occur. Hence $\theta_R'>0$ on $(0,R)$ and the Robin condition gives $\theta_R(R)=\beta$.

For every bounded Lipschitz open set $\Omega\subset\R^N$ satisfying the standing finite-component convention, let $\Omega_b^\sharp$ be the centered ball of equal $|x|^b$-weighted volume and define
\begin{equation*}
\lambda_{1,\beta}^{b}(\Omega)
:=
\inf_{\substack{u\in W^{1,p}(\Omega)\\ \int_\Omega|x|^b|u|^p\,dx>0}}
\frac{\displaystyle\int_\Omega|\nabla u|^p\,dx
+\beta\int_{\partial\Omega}|x|^{b/p'}|\operatorname{Tr}u|^p\,d\Haus^{N-1}}
{\displaystyle\int_\Omega|x|^b|u|^p\,dx}.
\end{equation*}
Combining the abstract transfer, the compatible power-weight geometry, spectral admissibility, and the radial flux analysis, we prove
\begin{equation*}
\boxed{\lambda_{1,\beta}^{b}(\Omega_b^\sharp)
\le
\lambda_{1,\beta}^{b}(\Omega)}
\end{equation*}
for every $\beta>0$ and
\[
N\ge2,\qquad 1<p\le N,\qquad -p<b<0.
\]

\subsection{Organization}
The paper is organized as follows. The next section introduces weighted volume and comparison balls. It is followed by the finite-exponent spectral-admissibility criterion, the geometric hypothesis, and the weighted transfer with its fixed-volume consequence. We then check consistency with the unweighted Lipschitz theory, verify the pure-power geometry, compare the result precisely with \cite{ACG}, and finally prove the radial flux result and the pure-power inequality.

\section{Weighted volume, comparison balls, and compatible perimeter}

Throughout the paper,
\[
N\ge2,\qquad 1<p<\infty,\qquad p'=\frac{p}{p-1}.
\]

\subsection{Notation and standing conventions}

We write $\mathcal L^k$ for $k$-dimensional Lebesgue measure and $\Haus^k$ for $k$-dimensional Hausdorff measure. In particular, $|E|:=\mathcal L^N(E)$ for a Lebesgue-measurable set $E\subset\R^N$. The open ball of radius $R>0$ centered at the origin is $B_R$, $\mathbb S^{N-1}:=\partial B_1$, and
\[
\omega_N:=|B_1|,
\qquad
\sigma_N:=\Haus^{N-1}(\mathbb S^{N-1})=N\omega_N.
\]
If $E$ and $F$ are sets, $\chi_E$ is the characteristic function of $E$ and $E\triangle F$ is their symmetric difference. Statements holding almost everywhere in $\Omega$ refer to $\mathcal L^N$, while almost-everywhere statements on $\partial\Omega$ refer to $\Haus^{N-1}$.

For a Radon measure $\nu$, the restriction of $\nu$ to a Borel set $A$ is denoted by $\nu\lfloor A$. If $F$ is Borel measurable, $F_\#\nu$ is the pushforward measure, defined by $(F_\#\nu)(A)=\nu(F^{-1}(A))$. The notation $d\mu=m\,dx$ means that $\mu$ is the Radon measure with density $m$ relative to $\mathcal L^N$; since $m\in L^1_{\mathrm{loc}}$, one has $\mu\ll\mathcal L^N$.

The spaces $L^q(\Omega)$, $W^{1,p}(\Omega)$, and $BV(\Omega)$ have their standard meanings; $L^q(\Omega,d\mu)$ denotes the corresponding space with respect to $\mu$. If $E$ has finite perimeter, $\partial^*E$ is its reduced boundary, $\nu_E$ is its measure-theoretic outward unit normal, $D\chi_E$ is the distributional derivative of $\chi_E$, and
\[
P(E):=|D\chi_E|(\R^N)=\Haus^{N-1}(\partial^*E)
\]
is its ordinary De Giorgi perimeter in $\R^N$. For a bounded Lipschitz domain $\Omega$, $\nu_\Omega$ is the outward unit normal and $\operatorname{Tr}u$ is the Sobolev or $BV$ trace of $u$ on $\partial\Omega$, as appropriate. The symbol $\widetilde u$ denotes the precise representative of a Sobolev function. Norms in $L^\infty$ are essential-supremum norms.

We write
\[
\Delta_pu:=\operatorname{div}(|\nabla u|^{p-2}\nabla u)
\]
for the $p$-Laplacian. For a classically differentiable function, $\partial_\nu u:=\nabla u\cdot\nu_\Omega$ is its outward normal derivative; boundary conditions in the weak setting are always understood through the displayed variational identity.

For an exponent $s\in(1,\infty)$, $s':=s/(s-1)$ is its H\"older conjugate. When $1<p<N$, we use
\[
p^*:=\frac{Np}{N-p},
\qquad
p_{\partial}:=\frac{p(N-1)}{N-p}
\]
for the Sobolev and trace critical exponents. A radial density is written $m(x)=\mathfrak m(|x|)$, with boundary profile $\mathfrak w(r):=\mathfrak m(r)^{1/p'}$. A radial function $z$ is written $z(x)=Z(|x|)$ when the distinction between the function on $\R^N$ and its profile is useful. A prime denotes differentiation with respect to the displayed radial variable. Constants denoted by $C$, possibly with subscripts, are positive and may change from line to line unless explicitly fixed.

\subsection{Radial weighted volume}
We first impose the standing condition on the radial volume density.
\begin{assumption}
\label{ass:radial-volume-density}
Let $m:\R^N\to[0,\infty]$ be radial and Borel measurable. We assume
\[
m\in L^1_{\mathrm{loc}}(\R^N)
\]
and that
\[
V(R):=\int_{B_R}m(x)\,dx
\]
is finite and strictly increasing for $R>0$, with
\[
V(0)=0,\qquad \lim_{R\to\infty}V(R)=\infty.
\]
\end{assumption}

Under \Cref{ass:radial-volume-density}, $V$ is continuous on $[0,\infty)$. Indeed, if $m(x)=\mathfrak m(|x|)$, then
\[
V(R)=\sigma_N\int_0^R\mathfrak m(r)r^{N-1}\,dr,
\]
which also proves the asserted continuity.
For a measurable set $E\subset\R^N$, define
\[
\mu(E):=\int_E m(x)\,dx.
\]
Then $\mu$ is a nonatomic Radon measure, finite on bounded sets.

The next observation explains how singular densities are interpreted.
\begin{remark}
The density may be singular on a Lebesgue-null set. In particular, $m_b(x)=|x|^b$, $-N<b<0$, may be assigned the value $+\infty$ at the origin without changing $\mu$.
\end{remark}

\subsection{Centered comparison balls}
Since $V$ is continuous and strictly increasing from $0$ to $\infty$, it has a continuous inverse
\[
\varrho:[0,\infty)\to[0,\infty),\qquad \varrho(0):=0,\quad \varrho(v):=V^{-1}(v)\ \ (v>0).
\]

The centered weighted rearrangement is defined through the preceding volume coordinate.
\begin{definition}
If $E\subset\R^N$ is measurable and $0<\mu(E)<\infty$, define
\[
E^\sharp:=B_{\varrho(\mu(E))}.
\]
Thus $\mu(E^\sharp)=\mu(E)$. For $v>0$, write
\[
B_v^\sharp:=B_{\varrho(v)}.
\]
Thus $\varrho=V^{-1}$ is the volume-to-radius function, whereas $R$ in $B_R$ denotes a fixed scalar radius.
\end{definition}

The volume-to-radius correspondence has the following elementary properties.
\begin{lemma}
\label{lem:radial-volume-coordinate}
Fix $R_0>0$ and define $\mathcal V(x):=V(|x|)$ on $B_{R_0}$. Then
\[
\mathcal V_\#(\mu\lfloor B_{R_0})=\mathcal L^1\lfloor(0,V(R_0)).
\]
Equivalently, for $0<s<V(R_0)$,
\[
\mu\bigl(\{x\in B_{R_0}:V(|x|)<s\}\bigr)=s.
\]
Consequently, for every nonnegative Borel function $F$,
\[
\int_{B_{R_0}}F(V(|x|))\,d\mu(x)
=
\int_0^{V(R_0)}F(s)\,ds.
\]
\end{lemma}

\begin{proof}
Since $V$ is strictly increasing,
\[
\{x\in B_{R_0}:V(|x|)<s\}=B_{\varrho(s)}.
\]
Hence its $\mu$-measure is $V(\varrho(s))=s$. The pushforward and integral identities follow.
\end{proof}

\begin{remark}
The coordinate $x\mapsto V(|x|)$ is uniformly distributed with respect to $d\mu$ on each centered ball. It is the ball-side measure-preserving coordinate used in the transfer argument.
\end{remark}

\subsection{Compatible boundary density}
The boundary density compatible with $m$ is defined as follows.
\begin{definition}
The boundary density associated with $m$ is
\[
w(x):=m(x)^{1/p'}.
\]
For a finite-perimeter set $E\subset\R^N$, define
\[
P_w(E):=\int_{\partial^*E}w(x)\,d\Haus^{N-1}(x),
\]
with the value $+\infty$ allowed.
\end{definition}

If $\Omega$ is a bounded Lipschitz domain, then
\[
P_w(\Omega)=\int_{\partial\Omega}w\,d\Haus^{N-1}
\]
whenever the right-hand side is finite.

The choice $w=m^{1/p'}$ is dictated by the coarea--H\"older estimate.
\begin{remark}
The exponent is dictated by the coarea--H\"older estimate \eqref{eq:intro-coarea-holder}. The same compatibility appears in radial flux identities through
\[
w^{p'}=m.
\]
\end{remark}

\subsection{Weighted coarea and centered perimeter profile}
We shall use the following weighted coarea formula.
\begin{lemma}
Let $O\subset\R^N$ be open, let $u\in BV(O)$, and let $g:O\to[0,\infty]$ be Borel measurable. Then
\[
\int_O g\,d|Du|
=
\int_{\R}
\left(\int_{\partial^*\{u>t\}\cap O}g\,d\Haus^{N-1}\right)dt,
\]
with both sides understood in $[0,\infty]$.
\end{lemma}

\begin{proof}
This is the standard coarea formula for $BV$ functions; see, for example, \cite{Maggi}. If $u\in W^{1,1}(O)$, then $d|Du|=|\nabla u|\,dx$.
\end{proof}

For $v>0$, define the centered weighted perimeter profile
\[
J(v):=P_w(B_{\varrho(v)}).
\]
A radial representative of $m$ and hence of $w$ is fixed throughout. In the concrete applications below, $w$ is continuous on $(0,\infty)$ and
\[
P_w(B_R)=\sigma_N R^{N-1}\mathfrak w(R).
\]

\section{Weighted Robin eigenvalues and spectral admissibility}
\label{sec:spectral}

Let
\[
N\ge2,
\qquad
1<p<\infty,
\qquad
\beta>0,
\]
and let $\Omega\subset\R^N$ be a bounded connected Lipschitz domain. Let
\[
m:\Omega\longrightarrow[0,\infty],
\qquad
w:\partial\Omega\longrightarrow[0,\infty)
\]
be measurable functions such that
\begin{equation}
m\ge0\quad\text{a.e. in }\Omega,
\qquad
w\ge0\quad\text{a.e. on }\partial\Omega,
\label{eq:spectral-basic}
\end{equation}
and
\begin{equation}
0<\int_\Omega m\,dx<\infty,
\qquad
0<\int_{\partial\Omega}w\,d\Haus^{N-1}<\infty.
\label{eq:spectral-positive-mass}
\end{equation}
Thus $m$ is finite almost everywhere, is not the zero element of $L^1(\Omega)$, and the constant function has positive weighted mass. Values assigned to $m$ on null sets are immaterial. Whenever the following expressions are finite, set
\[
\mathcal E_{\beta,\Omega}^{m,w}(u)
:=
\int_\Omega|\nabla u|^p\,dx
+
\beta\int_{\partial\Omega}
w|\operatorname{Tr}u|^p\,d\Haus^{N-1},
\]
and
\[
\mathcal M_\Omega^m(u)
:=
\int_\Omega m|u|^p\,dx.
\]
In the abstract definitions below, we assume that both forms are finite on $W^{1,p}(\Omega)$.

The weighted Rayleigh quotient and its first eigenvalue are defined next.
\begin{definition}
For $u\in W^{1,p}(\Omega)$ satisfying $\mathcal M_\Omega^m(u)>0$, define
\[
\mathcal R_{\beta,\Omega}^{m,w}(u)
:=
\frac{\mathcal E_{\beta,\Omega}^{m,w}(u)}
{\mathcal M_\Omega^m(u)}.
\]
The first weighted Robin eigenvalue is
\begin{equation}
\lambda_{1,\beta}^{m,w}(\Omega)
:=
\inf\left\{
\mathcal R_{\beta,\Omega}^{m,w}(u):
u\in W^{1,p}(\Omega),\ \mathcal M_\Omega^m(u)>0
\right\}.
\label{eq:weighted-rayleigh}
\end{equation}
\end{definition}

The abstract transfer argument uses the following notion of spectral admissibility.
\begin{definition}
\label{def:spectral-admissibility}
The pair $(m,w)$ is called spectrally admissible on $\Omega$ if:
\begin{enumerate}
\item the infimum in \eqref{eq:weighted-rayleigh} is attained by a function
\[
u\in W^{1,p}(\Omega)\cap L^\infty(\Omega);
\]
\item every nonnegative nontrivial first eigenfunction has a locally H\"older-continuous representative which is strictly positive in $\Omega$;
\item every first eigenfunction satisfies
\begin{align}
&\int_\Omega
|\nabla u|^{p-2}\nabla u\cdot\nabla\zeta\,dx
+
\beta\int_{\partial\Omega}
w|\operatorname{Tr}u|^{p-2}
\operatorname{Tr}u\,\operatorname{Tr}\zeta
\,d\Haus^{N-1}
\notag\\
&\hspace{35mm}=
\lambda_{1,\beta}^{m,w}(\Omega)
\int_\Omega
m|u|^{p-2}u\zeta\,dx
\label{eq:weak-eigenvalue}
\end{align}
for every $\zeta\in W^{1,p}(\Omega)$;
\item the positive first eigenfunction is unique up to multiplication by a positive constant.
\end{enumerate}
\end{definition}

\subsection{A finite-exponent sufficient framework}

For the remainder of this section, when proving a concrete sufficient criterion, assume
\[
1<p\le N.
\]
We impose
\begin{equation}
m\in L^q(\Omega)
\quad\text{for some }q>\frac Np,
\label{eq:m-finite-exponent}
\end{equation}
and
\begin{equation}
w\in L^r(\partial\Omega)
\quad\text{for some }r>\frac{N-1}{p-1}.
\label{eq:w-subcritical}
\end{equation}
Write
\[
q'=\frac q{q-1},
\qquad
r'=\frac r{r-1},
\]
with the conventions $q'=1$ when $q=\infty$ and $r'=1$ when $r=\infty$. If $p<N$, then
\begin{equation}
pq'<p^*,
\qquad
pr'<p_{\partial}.
\label{eq:strict-subcritical-exponents}
\end{equation}
If $p=N$, both $pq'$ and $pr'$ are finite. The Sobolev and trace embeddings are compact into $L^{pq'}(\Omega)$ and $L^{pr'}(\partial\Omega)$ in either case: for $p<N$ this follows from \eqref{eq:strict-subcritical-exponents}, while for $p=N$ it follows from compactness below every finite exponent.
H\"older's inequality and the Sobolev and trace embeddings give
\[
\int_{\partial\Omega}w|\operatorname{Tr}u|^p\,d\Haus^{N-1}
\le
\|w\|_{L^r(\partial\Omega)}
\|\operatorname{Tr}u\|_{L^{pr'}(\partial\Omega)}^p,
\]
and
\begin{equation}
\int_\Omega m|u|^p\,dx
\le
\|m\|_{L^q(\Omega)}
\|u\|_{L^{pq'}(\Omega)}^p.
\label{eq:weighted-mass-bound}
\end{equation}
Thus the abstract forms above are finite on $W^{1,p}(\Omega)$ under these finite-exponent hypotheses.

\subsection{Compactness and coercivity}

The exponent conditions above yield compactness of both weighted forms.
\begin{lemma}
\label{lem:weighted-form-compactness}
Suppose that $u_j\rightharpoonup u$ weakly in $W^{1,p}(\Omega)$. Then
\begin{equation}
\int_\Omega m|u_j|^p\,dx
\longrightarrow
\int_\Omega m|u|^p\,dx,
\label{eq:mass-compact}
\end{equation}
and
\begin{equation}
\int_{\partial\Omega}
w|\operatorname{Tr}u_j|^p\,d\Haus^{N-1}
\longrightarrow
\int_{\partial\Omega}
w|\operatorname{Tr}u|^p\,d\Haus^{N-1}.
\label{eq:boundary-compact}
\end{equation}
\end{lemma}

\begin{proof}
The compact Sobolev embedding described above gives
$u_j\to u$ strongly in $L^{pq'}(\Omega)$. Hence
$|u_j|^p\to|u|^p$ strongly in $L^{q'}(\Omega)$, and H\"older's inequality proves \eqref{eq:mass-compact}.
Similarly, compactness of the trace embedding implies
$\operatorname{Tr}u_j\to\operatorname{Tr}u$ strongly in $L^{pr'}(\partial\Omega)$. Thus
$|\operatorname{Tr}u_j|^p\to|\operatorname{Tr}u|^p$ strongly in $L^{r'}(\partial\Omega)$, which proves \eqref{eq:boundary-compact}.
\end{proof}

The same hypotheses also give a coercive control of the $W^{1,p}$-norm.
\begin{lemma}
\label{lem:compactness-coercivity}
There exists $C>0$ such that
\begin{equation}
\|u\|_{W^{1,p}(\Omega)}^p
\le
C\left[
\int_\Omega|\nabla u|^p\,dx
+
\int_{\partial\Omega}
w|\operatorname{Tr}u|^p\,d\Haus^{N-1}
\right]
\label{eq:weighted-coercivity}
\end{equation}
for every $u\in W^{1,p}(\Omega)$.
\end{lemma}

\begin{proof}
Suppose that \eqref{eq:weighted-coercivity} is false. Then there is a sequence $(u_j)$ such that
\[
\|u_j\|_{W^{1,p}(\Omega)}=1
\]
and the right-hand side of \eqref{eq:weighted-coercivity} tends to zero. Let
\[
c_j:=\frac1{|\Omega|}\int_\Omega u_j\,dx.
\]
The Poincar\'e inequality gives
$\|u_j-c_j\|_{W^{1,p}(\Omega)}\to0$. The constants $(c_j)$ are bounded because $(u_j)$ is bounded in $L^p(\Omega)$; after passing to a subsequence, $c_j\to c$. Hence $u_j\to c$ strongly in $W^{1,p}(\Omega)$ and in the trace space. Therefore
\[
\int_{\partial\Omega}w|\operatorname{Tr}u_j|^p\,d\Haus^{N-1}
\longrightarrow
|c|^p\int_{\partial\Omega}w\,d\Haus^{N-1}.
\]
By \eqref{eq:spectral-positive-mass}, $c=0$. This contradicts
$\|u_j\|_{W^{1,p}(\Omega)}=1$.
\end{proof}

\subsection{Existence and the Euler equation}

We first establish attainment and the corresponding Euler equation.
\begin{theorem}
\label{thm:attainment}
One has
\[
0<\lambda_{1,\beta}^{m,w}(\Omega)<\infty,
\]
and the infimum in \eqref{eq:weighted-rayleigh} is attained by a nonnegative function $u\in W^{1,p}(\Omega)$ satisfying
\begin{equation*}
\int_\Omega mu^p\,dx=1.
\end{equation*}
Every minimizer satisfies \eqref{eq:weak-eigenvalue}.
\end{theorem}

\begin{proof}
The constraint set
\[
\mathcal N:=\left\{u\in W^{1,p}(\Omega):\int_\Omega m|u|^p\,dx=1\right\}
\]
is nonempty by \eqref{eq:spectral-positive-mass}, since a suitable nonzero constant can be normalized to have weighted mass one. Let $(u_j)\subset\mathcal N$ be a minimizing sequence. By \Cref{lem:compactness-coercivity}, it is bounded in $W^{1,p}(\Omega)$. Passing to a subsequence, $u_j\rightharpoonup u$ weakly in $W^{1,p}(\Omega)$. By \Cref{lem:weighted-form-compactness}, $u\in\mathcal N$. Weak lower semicontinuity of the gradient term and compact convergence of the boundary term show that $u$ is a minimizer. Replacing $u$ by $|u|$ gives a nonnegative minimizer.

To prove positivity of the eigenvalue, observe that
\[
\mathcal E_{\beta,\Omega}^{m,w}(u)
\ge
\min\{1,\beta\}
\left[
\int_\Omega|\nabla u|^p\,dx
+
\int_{\partial\Omega}w|\operatorname{Tr}u|^p\,d\Haus^{N-1}
\right].
\]
Combining this with \eqref{eq:weighted-coercivity} and \eqref{eq:weighted-mass-bound} yields a positive lower bound for the energy on $\mathcal N$. Finiteness follows by evaluating the quotient on any function with positive weighted mass.

Finally, for $G(u):=\int_\Omega m|u|^p\,dx$, one has $G'(u)[u]=p$ on $\mathcal N$. Thus the constraint derivative does not vanish there. The Lagrange multiplier rule gives \eqref{eq:weak-eigenvalue}, with multiplier $\lambda_{1,\beta}^{m,w}(\Omega)$ after testing the Euler equation with $u$.
\end{proof}

\subsection{Boundedness and positivity}

The minimizer obtained above is essentially bounded.
\begin{theorem}
\label{thm:first-eigenfunction-boundedness}
Every nonnegative first eigenfunction belongs to $L^\infty(\Omega)$.
\end{theorem}

\begin{proof}
Let $u\ge0$ be a first eigenfunction, and write $\lambda:=\lambda_{1,\beta}^{m,w}(\Omega)$. For $L>0$, set $u_L:=\min\{u,L\}$. Fix $k\ge1$ and use
\[
\zeta=u\,u_L^{p(k-1)}
\]
as a test function in \eqref{eq:weak-eigenvalue}. The gradient term equals
\begin{align*}
&\int_{\{u\ge L\}}L^{p(k-1)}|\nabla u|^p\,dx
+
\bigl[1+p(k-1)\bigr]
\int_{\{u<L\}}u^{p(k-1)}|\nabla u|^p\,dx.
\end{align*}
Since the Robin term is nonnegative, this implies
\begin{equation*}
\int_\Omega
\left|\nabla\bigl(u\,u_L^{k-1}\bigr)\right|^p\,dx
\le
C_pk^p\lambda
\int_\Omega
m\bigl(u\,u_L^{k-1}\bigr)^p\,dx.
\end{equation*}
Put $v_L:=u\,u_L^{k-1}$. H\"older's inequality gives
\begin{equation*}
\int_\Omega mv_L^p\,dx
\le
\|m\|_{L^q(\Omega)}
\|v_L\|_{L^{pq'}(\Omega)}^p.
\end{equation*}
Choose
\[
s:=p^*\quad\text{if }p<N,
\]
whereas, if $p=N$, fix any finite $s>pq'$. In both cases $s>pq'$ and the Sobolev inequality on a bounded Lipschitz domain gives
\[
\|v_L\|_{L^s(\Omega)}^p
\le
C\left(
\|\nabla v_L\|_{L^p(\Omega)}^p
+
\|v_L\|_{L^p(\Omega)}^p
\right).
\]
Since $pq'\ge p$ and $|\Omega|<\infty$,
$\|v_L\|_{L^p}\le C_\Omega\|v_L\|_{L^{pq'}}$. Combining these estimates yields
\begin{equation*}
\|v_L\|_{L^s(\Omega)}
\le
C_0^{1/p}k\,
\|v_L\|_{L^{pq'}(\Omega)},
\end{equation*}
where $C_0$ is independent of $L$ and $k$. For $k=1$ the norm on the right is finite by the Sobolev embedding. Thereafter the exponents below are chosen inductively, so the right-hand side at each finite step is finite by the preceding step. Monotone convergence as $L\to\infty$ gives
\begin{equation}
\|u\|_{L^{ks}(\Omega)}
\le
\bigl(C_0^{1/p}k\bigr)^{1/k}
\|u\|_{L^{kpq'}(\Omega)}.
\label{eq:moser-step}
\end{equation}
Set
\[
\chi:=\frac{s}{pq'}>1,
\qquad
k_j:=\chi^j.
\]
Then $k_js=k_{j+1}pq'$, and iteration of \eqref{eq:moser-step} gives
\[
\|u\|_{L^{k_{J+1}pq'}(\Omega)}
\le
\left[
\prod_{j=0}^J
\bigl(C_0^{1/p}k_j\bigr)^{1/k_j}
\right]
\|u\|_{L^{pq'}(\Omega)}.
\]
The infinite product converges because
$\sum_j(1+\log k_j)/k_j<\infty$. Since $k_{J+1}pq'\to\infty$, one obtains $u\in L^\infty(\Omega)$.
\end{proof}

The nonnegative first eigenfunction is in fact continuous and strictly positive in the interior.
\begin{theorem}
\label{thm:first-eigenfunction-positivity}
Every nonnegative first eigenfunction has a locally H\"older-continuous representative and satisfies
\begin{equation*}
u>0\qquad\text{in }\Omega.
\end{equation*}
\end{theorem}

\begin{proof}
By \Cref{thm:first-eigenfunction-boundedness}, $u\in L^\infty(\Omega)$. The interior equation is
\[
-\Delta_pu=\lambda_{1,\beta}^{m,w}(\Omega)mu^{p-1}.
\]
The right-hand side belongs to $L^q_{\mathrm{loc}}(\Omega)$ with $q>N/p$. Choose $\varepsilon\in(0,p)$ so small that $q>N/(p-\varepsilon)$. The special case of Serrin's local continuity theorem \cite[Theorem~8]{Serrin} for $A(\xi)=|\xi|^{p-2}\xi$ and datum $f=\lambda m u^{p-1}$ then gives a representative in $C^{0,\alpha}_{\mathrm{loc}}(\Omega)$ for some $\alpha\in(0,1)$.
Moreover, $u$ is a nonnegative weak supersolution of $-\Delta_pu\ge0$. Trudinger's weak Harnack inequality \cite[Theorem~1.2]{Trudinger} states that, whenever $B_{2\rho}(x_0)\Subset\Omega$, there are $\kappa>0$ and $C>0$, depending only on $N$, $p$, and the fixed radius ratio, such that
\[
\left(\frac1{|B_{2\rho}|}\int_{B_{2\rho}(x_0)}u^\kappa\,dx\right)^{1/\kappa}
\le C\,\operatorname*{ess\,inf}_{B_\rho(x_0)}u.
\]
If the continuous representative vanished at $x_0$, the essential infimum on every sufficiently small $B_\rho(x_0)$ would be zero, so the inequality would force $u=0$ almost everywhere on $B_{2\rho}(x_0)$ and hence everywhere there by continuity. Thus its zero set is both open and closed in the connected set $\Omega$. It cannot be all of $\Omega$ because the eigenfunction is nontrivial. Therefore $u>0$ throughout $\Omega$.
\end{proof}

\subsection{Simplicity}

The first eigenspace is one-dimensional in the positive cone.
\begin{theorem}
\label{thm:first-eigenvalue-simplicity}
Any two positive first eigenfunctions are proportional.
\end{theorem}

\begin{proof}
Let $u,v>0$ be first eigenfunctions corresponding to
$\lambda:=\lambda_{1,\beta}^{m,w}(\Omega)$. For $\varepsilon>0$, define
\[
\psi_\varepsilon:=\frac{u^p}{(v+\varepsilon)^{p-1}}.
\]
Because $u,v\in W^{1,p}(\Omega)\cap L^\infty(\Omega)$ and $v+\varepsilon\ge\varepsilon$, one has $\psi_\varepsilon\in W^{1,p}(\Omega)$. Moreover, the composition rule for Sobolev traces gives
\[
\operatorname{Tr}\psi_\varepsilon
=
\frac{|\operatorname{Tr}u|^p}
{(\operatorname{Tr}v+\varepsilon)^{p-1}}.
\]
Testing the equation for $v$ with $\psi_\varepsilon$ and subtracting it from the equation for $u$ tested with $u$ gives
\begin{align*}
&\int_\Omega
\left[
|\nabla u|^p
-
|\nabla v|^{p-2}\nabla v\cdot
\nabla\left(\frac{u^p}{(v+\varepsilon)^{p-1}}\right)
\right]dx
\\
&\quad+
\beta\int_{\partial\Omega}
w|\operatorname{Tr}u|^p
\left[
1-
\left(\frac{\operatorname{Tr}v}{\operatorname{Tr}v+\varepsilon}\right)^{p-1}
\right]d\Haus^{N-1}
\\
&=
\lambda\int_\Omega
mu^p
\left[
1-
\left(\frac{v}{v+\varepsilon}\right)^{p-1}
\right]dx.
\end{align*}
Picone's inequality applied to the pair $(u,v+\varepsilon)$ shows pointwise that the first integrand is nonnegative; separately, the boundary integrand is nonnegative. Consequently the first integral is bounded above by the right-hand side, which tends to zero by dominated convergence. Hence the integral of the Picone integrand tends to zero. At almost every point of $\Omega$, where $u$, $v$, and their weak gradients have pointwise representatives, that integrand converges to the nonnegative algebraic Picone integrand
\begin{equation*}
L_0:=
|\nabla u|^p
-
|\nabla v|^{p-2}\nabla v\cdot
\nabla\left(\frac{u^p}{v^{p-1}}\right).
\end{equation*}
Here the displayed derivative is interpreted locally: on every compact subset of $\Omega$, positivity and continuity give a positive lower bound for $v$, so $u^p/v^{p-1}\in W^{1,p}_{\mathrm{loc}}(\Omega)$. Fatou's lemma now gives $\int_\Omega L_0\,dx=0$. Hence $L_0=0$ almost everywhere. The pointwise equality case in the Allegretto--Huang Picone identity \cite{AllegrettoHuang} yields $\nabla(u/v)=0$ almost everywhere. Since $u/v\in W^{1,p}_{\mathrm{loc}}(\Omega)$ and $\Omega$ is connected, $u=cv$ for some $c>0$; interior continuity makes the proportionality pointwise.
\end{proof}

\subsection{Spectral admissibility criterion}

\begin{theorem}[Finite-exponent spectral admissibility]
\label{thm:spectral-admissibility}
Let $1<p\le N$. Under \eqref{eq:spectral-basic}, \eqref{eq:spectral-positive-mass}, \eqref{eq:m-finite-exponent}, and \eqref{eq:w-subcritical}, the pair $(m,w)$ is spectrally admissible on $\Omega$.
\end{theorem}

\begin{proof}
Existence and the weak equation follow from \Cref{thm:attainment}. Boundedness follows from \Cref{thm:first-eigenfunction-boundedness}, positivity from \Cref{thm:first-eigenfunction-positivity}, and simplicity from \Cref{thm:first-eigenvalue-simplicity}. These are exactly the requirements in \Cref{def:spectral-admissibility}.
\end{proof}

\subsection{The compatible pure-power pair}

The general criterion immediately gives spectral admissibility for the compatible power pair.
\begin{corollary}
\label{cor:power-spectral}
Let
\[
N\ge2,
\qquad
1<p\le N,
\qquad
-p<b<0,
\]
and let $\Omega\subset\R^N$ be a bounded connected Lipschitz domain. Define finite representatives at the origin by
\[
m_b(0)=w_b(0)=1,
\]
and, for $x\ne0$, set
\begin{equation*}
m_b(x)=|x|^b,
\qquad
w_b(x)=|x|^{b/p'}.
\end{equation*}
Then $(m_b,w_b)$ is spectrally admissible on $\Omega$. Moreover, $w_b\in L^1(\partial\Omega)$.
\end{corollary}

\begin{proof}
Because $b>-p$, one may choose $q$ such that
\[
\frac Np<q<-\frac Nb.
\]
Then $bq>-N$, and hence $|x|^b\in L^q(\Omega)$.
For the boundary density, choose $r$ satisfying
\[
\frac{N-1}{p-1}<r<-\frac{p'(N-1)}b.
\]
Such an $r$ exists precisely because $b>-p$. To verify boundary integrability near a possible contact point $0\in\partial\Omega$, let
\[
A_j:=\partial\Omega\cap\{2^{-j-1}<|x|\le2^{-j}\}.
\]
To see the required upper Ahlfors estimate directly, cover the compact Lipschitz boundary by finitely many graph charts. The area formula in each chart gives
$\Haus^{N-1}(\partial\Omega\cap B_\rho(x))\le C\rho^{N-1}$, with a uniform $C$ after taking the maximum over the finite atlas. Therefore
$\Haus^{N-1}(A_j)\le C2^{-j(N-1)}$, and
\[
\int_{\partial\Omega\cap B_1}|x|^{br/p'}\,d\Haus^{N-1}
\le
C\sum_{j=0}^\infty
2^{-j(N-1+br/p')},
\]
which converges because $N-1+br/p'>0$. Away from the origin the density is bounded. Thus $w_b\in L^r(\partial\Omega)\subset L^1(\partial\Omega)$, and the conclusion follows from \Cref{thm:spectral-admissibility}.
\end{proof}

The threshold $b=-p$ separates the admissible range from the Hardy-critical behavior.
\begin{proposition}
\label{prop:hardy-threshold}
Assume $1<p\le N$, $0\in\Omega$, and consider the interior mass $|x|^b\,dx$. If $p<N$ and $-N<b<-p$, then the infimum of the corresponding Robin quotient is zero. If $p<N$ and $b=-p$, the weighted mass embedding is not compact. If $p=N$, the endpoint $b=-p=-N$ is not locally integrable and belongs to a different, logarithmic critical problem.
\end{proposition}

\begin{proof}
Choose $\eta\in C_c^\infty(B_1\setminus\overline{B_{1/2}})$, $\eta\ne0$, and, for sufficiently small $\varepsilon$, set $u_\varepsilon(x)=\eta(x/\varepsilon)$. Its boundary term vanishes, while
\[
\frac{\int_\Omega|\nabla u_\varepsilon|^p\,dx}
{\int_\Omega|x|^b|u_\varepsilon|^p\,dx}
=C_\eta\varepsilon^{-p-b}.
\]
This tends to zero when $b<-p$. At $b=-p$ and $p<N$, the normalized family
\[
v_\varepsilon(x):=\varepsilon^{-(N-p)/p}\eta(x/\varepsilon)
\]
has both $\int|\nabla v_\varepsilon|^p$ and $\int|x|^{-p}|v_\varepsilon|^p$ independent of $\varepsilon$, but $v_\varepsilon\rightharpoonup0$ in $W^{1,p}$; hence the mass embedding is not compact. The classical Hardy inequality
\[
\left(\frac{N-p}{p}\right)^p
\int_{\R^N}|x|^{-p}|v|^p\,dx
\le
\int_{\R^N}|\nabla v|^p\,dx,
\qquad v\in C_c^\infty(\R^N),
\]
identifies the same critical scaling. For $p=N$, $|x|^{-N}\notin L^1_{\mathrm{loc}}$.
\end{proof}

\begin{remark}
For $b<0$, the inequalities needed to choose the spectral exponents $q$ and $r$ in \Cref{cor:power-spectral} are both equivalent to $b>-p$. The same condition is $\delta=1+b/p>0$ in the radial ODE and makes the dilation factors $q^{-p-b}$ and $q^{-1-b/p}$ strictly smaller than one. By contrast, the geometric reduction only requires $b>-N$ when $1<p\le N$.
\end{remark}

On centered balls, the first eigenfunction is radial.
\begin{corollary}
\label{cor:radiality-ball}
Let $m$ and $w$ be radial and satisfy the hypotheses of \Cref{thm:spectral-admissibility} on a centered ball $B_R$. Then the positive first eigenfunction on $B_R$ is radial.
\end{corollary}

\begin{proof}
Let $z$ be the positive first eigenfunction, normalized by
$\int_{B_R}mz^p\,dx=1$. For every orthogonal transformation $O$, the function $z_O(x):=z(Ox)$ has the same weighted mass and weighted Robin energy as $z$. Hence $z_O$ is also a positive normalized first eigenfunction. By simplicity, $z_O=c_Oz$ for some $c_O>0$; the common normalization gives $c_O=1$. Therefore $z(Ox)=z(x)$ for every orthogonal $O$, and $z$ is radial.
\end{proof}

\begin{remark}
When $p<N$, the strict inequalities $q>N/p$ and $r>(N-1)/(p-1)$ place the weighted mass and boundary forms below the critical Sobolev and trace exponents. When $p=N$, all finite target exponents lie below the limiting Sobolev and trace embeddings. Thus the same direct compactness and iteration argument covers the endpoint $p=N$.
\end{remark}

\section{The weighted double-density isoperimetric principle}

Let $m$ satisfy \Cref{ass:radial-volume-density}, set $d\mu=m(x)\,dx$ and $w=m^{1/p'}$, and recall
\[
J(v):=P_w(B_v^\sharp).
\]

The geometric input required by the transfer is the following compatible double-density inequality.
\begin{assumption}
\label{ass:double-density}
For every $v>0$,
\[
0<J(v)<\infty.
\]
Moreover, every bounded finite-perimeter set $E\subset\R^N$ with $0<\mu(E)<\infty$ satisfies
\begin{equation}
P_w(E)\ge P_w(E^\sharp)=J(\mu(E)).
\label{eq:double-density}
\end{equation}
The inequality is understood in the extended sense.
\end{assumption}

\begin{remark}
The two densities are compatible rather than independent: $m$ determines weighted volume, whereas $w=m^{1/p'}$ determines perimeter. No homogeneity assumption is made. Only bounded sets are required because the transfer proof applies the inequality to superlevel sets inside a bounded domain.
\end{remark}

Applied to superlevel sets, the preceding assumption gives the perimeter comparison used below.
\begin{proposition}
\label{prop:superlevel-isop}
Suppose that \Cref{ass:double-density} holds. Let $U_t=\{x\in\Omega:u(x)>t\}$ have finite perimeter and positive weighted volume, and let $r(t)>0$ satisfy
\[
\mu(B_{r(t)})=\mu(U_t).
\]
Then
\[
P_w(U_t)\ge P_w(B_{r(t)}).
\]
\end{proposition}

\begin{proof}
Since $U_t^\sharp=B_{r(t)}$, this is exactly \eqref{eq:double-density}.
\end{proof}

\section{The weighted Bossel--Daners transfer}

Throughout this section,
\[
N\ge2,
\qquad
1<p<\infty,
\qquad
p'=\frac{p}{p-1},
\qquad
\beta>0.
\]
In addition to \Cref{ass:radial-volume-density}, we assume that the fixed radial representative
$m(x)=\mathfrak m(|x|)$ satisfies
\begin{equation}
0<\mathfrak m(r)<\infty
\qquad\text{for every }r>0.
\label{eq:positive-radial-density}
\end{equation}
The value at the origin is irrelevant. Set $w=m^{1/p'}$. Thus $w$ is finite and strictly positive on $\R^N\setminus\{0\}$.

Let $\Omega\subset\R^N$ be a bounded connected Lipschitz domain with
$0<\mu(\Omega)<\infty$, and assume
\begin{equation}
\int_{\partial\Omega}w\,d\Haus^{N-1}<\infty.
\label{eq:boundary-weight-integrable}
\end{equation}
Assume that $(m,w)$ is spectrally admissible on $\Omega$, and let $u$ be its positive first eigenfunction, normalized by
\begin{equation*}
\|u\|_{L^\infty(\Omega)}=1.
\end{equation*}
We use the precise Sobolev representative $\widetilde u$ in the interior and the Sobolev trace $\operatorname{Tr}u$ on $\partial\Omega$.

\subsection{Weighted level-set functional}

For $0<t<1$, define
\begin{equation*}
U_t:=\{x\in\Omega:\widetilde u(x)>t\},
\qquad
\Gamma_t:=\{x\in\partial\Omega:\operatorname{Tr}u(x)>t\}.
\end{equation*}
For almost every $t$, $U_t$ has finite perimeter in $\Omega$. At such a level, write
\begin{equation*}
\Sigma_t:=\partial^*U_t\cap\Omega.
\end{equation*}
For a nonnegative Borel function $\varphi$, define
\begin{equation*}
\mathcal H_\Omega^w(U_t,\varphi)
:=
\frac{1}{\mu(U_t)}
\left[
\int_{\Sigma_t}\varphi w\,d\Haus^{N-1}
+
\beta\int_{\Gamma_t}w\,d\Haus^{N-1}
-
(p-1)\int_{U_t}\varphi^{p'}\,d\mu
\right]
\end{equation*}
whenever the terms on the right are finite.

The eigenvalue admits the following weighted level-set representation.
\begin{lemma}
\label{lem:level-set-identity}
Define
\begin{equation*}
\varphi_u(x):=
\frac{|\nabla u(x)|^{p-1}}
{w(x)u(x)^{p-1}}
\end{equation*}
at points $x\ne0$ where the Sobolev gradient is defined. Set $\varphi_u=0$ at the origin and on the negligible exceptional set where the quotient is undefined. Then, for almost every $t\in(0,1)$,
\begin{equation*}
\lambda_{1,\beta}^{m,w}(\Omega)
=
\mathcal H_\Omega^w(U_t,\varphi_u).
\end{equation*}
\end{lemma}

\begin{proof}
Fix $t\in(0,1)$ and $0<\varepsilon<t$. Define the Lipschitz function
\[
h_{\varepsilon,t}(s):=
\begin{cases}
0,&0\le s\le t,\\
s^{1-p}(s-t)/\varepsilon,&t<s<t+\varepsilon,\\
s^{1-p},&s\ge t+\varepsilon,
\end{cases}
\]
and set
\[
\psi_{\varepsilon,t}:=h_{\varepsilon,t}(u).
\]
Then $\psi_{\varepsilon,t}\in W^{1,p}(\Omega)$ and is admissible in the weak eigenvalue equation. This piecewise definition also fixes its value on $\{u=0\}$. By the Sobolev chain rule, its contribution to the gradient term is
\begin{align*}
I_{\varepsilon,t}
:={}&
\int_\Omega
|\nabla u|^{p-2}\nabla u\cdot\nabla\psi_{\varepsilon,t}\,dx
\\
={}&
\frac1\varepsilon
\int_{\{t<u<t+\varepsilon\}}
u^{1-p}|\nabla u|^p\,dx
-
\frac{p-1}{\varepsilon}
\int_{\{t<u<t+\varepsilon\}}
(u-t)u^{-p}|\nabla u|^p\,dx
\\
&-
(p-1)
\int_{\{u\ge t+\varepsilon\}}
u^{-p}|\nabla u|^p\,dx.
\end{align*}
For the first term, the coarea formula gives
\[
\frac1\varepsilon
\int_t^{t+\varepsilon}
s^{1-p}
\left(
\int_{\Sigma_s}|\nabla u|^{p-1}\,d\Haus^{N-1}
\right)ds.
\]
Lebesgue differentiation therefore identifies its limit for almost every $t$. The absolute value of the second term is bounded by
\[
(p-1)t^{-p}
\int_{\{t<u<t+\varepsilon\}}|\nabla u|^p\,dx,
\]
which tends to zero by absolute continuity of the integral. The last term converges by monotone convergence. Consequently, for almost every $t$,
\begin{align*}
\lim_{\varepsilon\downarrow0}
I_{\varepsilon,t}
={}&
\int_{\Sigma_t}
\frac{|\nabla u|^{p-1}}{u^{p-1}}\,d\Haus^{N-1}
-(p-1)
\int_{U_t}\frac{|\nabla u|^p}{u^p}\,dx.
\end{align*}
On the boundary,
$(\operatorname{Tr}u)^{p-1}\operatorname{Tr}\psi_{\varepsilon,t}
\to\chi_{\Gamma_t}$ almost everywhere and is bounded by $1$. Hence
\eqref{eq:boundary-weight-integrable} and dominated convergence yield
\[
\beta\int_{\partial\Omega}
w(\operatorname{Tr}u)^{p-1}\operatorname{Tr}\psi_{\varepsilon,t}\,d\Haus^{N-1}
\longrightarrow
\beta\int_{\Gamma_t}w\,d\Haus^{N-1}.
\]
The right-hand side of the eigenvalue equation converges to
$\lambda_{1,\beta}^{m,w}(\Omega)\mu(U_t)$. Finally,
\[
\varphi_uw=\frac{|\nabla u|^{p-1}}{u^{p-1}},
\qquad
\varphi_u^{p'}\,d\mu=\frac{|\nabla u|^p}{u^p}\,dx,
\]
because $w^{p'}=m$. Division by $\mu(U_t)$ proves the result.
\end{proof}

\subsection{The weighted selection lemma}

The next lemma selects a level at which the comparison function does not increase the level-set functional.
\begin{lemma}
\label{lem:selection}
Let $\varphi:\Omega\to[0,\infty)$ be Borel measurable and satisfy
\begin{equation*}
\varphi\in L^{p'}(\Omega,d\mu).
\end{equation*}
For $t\in(0,1)$, define
\begin{equation}
F(t):=
\int_{U_t}
(\varphi-\varphi_u)
\frac{|\nabla u|}{u}w\,dx.
\label{eq:F-def}
\end{equation}
Then $F(t)$ is finite for every $t>0$ and is locally absolutely continuous on $(0,1)$. Moreover, for almost every $t\in(0,1)$,
\begin{equation}
\mathcal H_\Omega^w(U_t,\varphi)
\le
\lambda_{1,\beta}^{m,w}(\Omega)
-
\frac{1}{\mu(U_t)t^{p-1}}
\frac{d}{dt}\bigl(t^pF(t)\bigr).
\label{eq:selection-differential}
\end{equation}
Consequently, given any full-measure subset $\mathcal T\subset(0,1)$, there exists $t\in\mathcal T$ such that
\begin{equation}
\mathcal H_\Omega^w(U_t,\varphi)
\le
\lambda_{1,\beta}^{m,w}(\Omega).
\label{eq:selection-conclusion}
\end{equation}
\end{lemma}

\begin{proof}
For $t>0$, H\"older's inequality and $w^{p'}=m$ give
\[
\int_{U_t}\varphi\frac{|\nabla u|}{u}w\,dx
\le
\left(\int_{U_t}\varphi^{p'}\,d\mu\right)^{1/p'}
\left(\int_{U_t}\frac{|\nabla u|^p}{u^p}\,dx\right)^{1/p}<\infty.
\]
Also,
\[
\int_{U_t}\varphi_u\frac{|\nabla u|}{u}w\,dx
=
\int_{U_t}\frac{|\nabla u|^p}{u^p}\,dx<\infty.
\]
Thus $F(t)$ is well defined.

By \Cref{lem:level-set-identity},
\begin{align*}
\mu(U_t)
\bigl(\mathcal H_\Omega^w(U_t,\varphi)-\lambda_{1,\beta}^{m,w}(\Omega)\bigr)
={}&
\int_{\Sigma_t}(\varphi-\varphi_u)w\,d\Haus^{N-1}
-(p-1)\int_{U_t}(\varphi^{p'}-\varphi_u^{p'})\,d\mu.
\end{align*}
Convexity gives
\[
(p-1)(\varphi^{p'}-\varphi_u^{p'})
\ge
p\varphi_u^{p'-1}(\varphi-\varphi_u),
\]
and
\[
\varphi_u^{p'-1}\,d\mu
=
\frac{|\nabla u|}{u}w\,dx.
\]
Hence
\[
\mu(U_t)
\bigl(\mathcal H_\Omega^w(U_t,\varphi)-\lambda_{1,\beta}^{m,w}(\Omega)\bigr)
\le
\int_{\Sigma_t}(\varphi-\varphi_u)w\,d\Haus^{N-1}-pF(t).
\]
On compact subintervals of $(0,1)$ the integrand in \eqref{eq:F-def} is absolutely integrable. The coarea formula yields
\[
F(t)=
\int_t^1\frac1\tau
\int_{\Sigma_\tau}(\varphi-\varphi_u)w\,d\Haus^{N-1}\,d\tau,
\]
so
\[
F'(t)=-\frac1t
\int_{\Sigma_t}(\varphi-\varphi_u)w\,d\Haus^{N-1}
\]
for almost every $t$. Therefore
\[
\int_{\Sigma_t}(\varphi-\varphi_u)w\,d\Haus^{N-1}-pF(t)
=
-t^{1-p}\frac{d}{dt}\bigl(t^pF(t)\bigr),
\]
which proves \eqref{eq:selection-differential}.

Suppose that \eqref{eq:selection-conclusion} fails on a full-measure set. Then
$G(t):=t^pF(t)$ is strictly decreasing. For any fixed $t_0>0$, the integrand in \eqref{eq:F-def} is absolutely integrable on $U_{t_0}$, while
$U_t\downarrow\{u=1\}$ as $t\uparrow1$ and $\nabla u=0$ almost everywhere on $\{u=1\}$. Dominated convergence gives $G(t)\to0$ as $t\uparrow1$, so $G(t)>0$ for $0<t<1$. On the other hand,
\begin{align*}
0<G(t)
&=t^pF(t)
\le
t^p\int_{U_t}\varphi\frac{|\nabla u|}{u}w\,dx
\le
t^{p-1}
\|\varphi\|_{L^{p'}(\Omega,d\mu)}
\|\nabla u\|_{L^p(\Omega)}
\longrightarrow0
\end{align*}
as $t\downarrow0$. This contradicts strict decrease. Hence \eqref{eq:selection-conclusion} holds at a level in every prescribed full-measure set.
\end{proof}

\subsection{Superlevel sets and the boundary trace}

We first record the exact weighted perimeter identity for the zero extension of a Sobolev function.
\begin{lemma}
\label{lem:weighted-superlevel-perimeter}
For almost every $t\in(0,1)$, the zero extension $\widehat\chi_{U_t}$ of $\chi_{U_t}$ belongs to $BV(\R^N)$ and
\begin{equation}
P_w(U_t)
=
\int_{\Sigma_t}w\,d\Haus^{N-1}
+
\int_{\Gamma_t}w\,d\Haus^{N-1}.
\label{eq:superlevel-perimeter}
\end{equation}
All terms are understood in $[0,\infty]$; in fact, they are finite for almost every $t$.
\end{lemma}

\begin{proof}
Since $u\in W^{1,p}(\Omega)\subset W^{1,1}(\Omega)$ and $u\ge0$, its zero extension $\widehat u$ belongs to $BV(\R^N)$. The gluing formula on a Lipschitz domain gives the measure identity
\begin{equation}
|D\widehat u|
=|\nabla u|\,\mathcal L^N\lfloor\Omega
+
\operatorname{Tr}u\,\Haus^{N-1}\lfloor\partial\Omega.
\label{eq:zero-extension-u}
\end{equation}
For almost every $t>0$, the set
$E_t:=\{\widehat u>t\}$ has finite perimeter in $\R^N$, agrees with $U_t$ in $\Omega$, and is empty outside $\overline\Omega$. At a boundary point at which the interior trace of $u$ exists, the half-ball density characterization of the trace shows
\begin{equation}
\partial^*E_t\cap\partial\Omega
\subset
\{\operatorname{Tr}u\ge t\}
\quad\text{up to an }\Haus^{N-1}\text{-null set}.
\label{eq:trace-inclusion}
\end{equation}
The possible equality set causes no difficulty: by Fubini, $\Haus^{N-1}(\{\operatorname{Tr}u=t\})=0$ for almost every $t$.

Indeed, apply the $BV$ coarea formula to $\widehat u$ and restrict the resulting measures to $\partial\Omega$. From \eqref{eq:zero-extension-u}, for every Borel set $A\subset\partial\Omega$,
\begin{align*}
\int_A\operatorname{Tr}u\,d\Haus^{N-1}
&=|D\widehat u|(A)
=\int_0^\infty
\Haus^{N-1}(A\cap\partial^*E_t)\,dt.
\end{align*}
On the other hand, the layer-cake formula gives
\[
\int_A\operatorname{Tr}u\,d\Haus^{N-1}
=\int_0^\infty
\Haus^{N-1}(A\cap\{\operatorname{Tr}u>t\})\,dt.
\]
Together with \eqref{eq:trace-inclusion}, equality of these nonnegative integrated measures implies
\begin{equation*}
\partial^*E_t\cap\partial\Omega
=\{\operatorname{Tr}u>t\}
\quad\text{up to }\Haus^{N-1}\text{-null sets}
\end{equation*}
for almost every $t>0$. Equivalently,
$\operatorname{Tr}\chi_{U_t}=\chi_{\{\operatorname{Tr}u>t\}}$ for almost every $t$. The zero-extension formula applied now to $\chi_{U_t}$ gives
\[
D\widehat\chi_{U_t}
=D\chi_{U_t}\lfloor\Omega
-
\chi_{\Gamma_t}\nu_\Omega\,
\Haus^{N-1}\lfloor\partial\Omega.
\]
The two measures on the right are mutually singular. Their total variations therefore add, and integration of the nonnegative Borel weight $w$ proves \eqref{eq:superlevel-perimeter}.

Finally,
\[
\int_\Omega w|\nabla u|\,dx
\le
\mu(\Omega)^{1/p'}\|\nabla u\|_{L^p(\Omega)}<\infty,
\]
so the weighted coarea formula gives
$\int_{\Sigma_t}w\,d\Haus^{N-1}<\infty$ for almost every $t$. The boundary term is finite by \eqref{eq:boundary-weight-integrable}.
\end{proof}

\subsection{Regular radial flux on the comparison ball}

Let $B_R=\Omega^\sharp$, and assume
\begin{equation}
\int_{\partial B_R}w\,d\Haus^{N-1}<\infty.
\label{eq:ball-boundary-weight}
\end{equation}
Assume that $(m,w)$ is spectrally admissible on $B_R$, and let $z$ be the positive first eigenfunction on $B_R$, normalized by
$\|z\|_{L^\infty(B_R)}=1$.

The comparison-ball eigenfunction is required to have the following radial structure.
\begin{assumption}
\label{ass:flux}
The eigenfunction has a radial representative $z=z(r)>0$ which is locally absolutely continuous on $(0,R]$ and strictly decreasing. Define
\begin{equation}
\mathfrak f_z(r)
:=-r^{N-1}|z'(r)|^{p-2}z'(r)
=r^{N-1}|z'(r)|^{p-1}.
\label{eq:radial-flux-def}
\end{equation}
We assume
\begin{equation}
\mathfrak f_z\in W^{1,1}_{\mathrm{loc}}((0,R]),
\qquad
\mathfrak f_z'(r)
=
\lambda_{1,\beta}^{m,w}(B_R)
r^{N-1}\mathfrak m(r)z(r)^{p-1}
\label{eq:radial-flux-equation}
\end{equation}
for almost every $r\in(0,R)$, and
\begin{equation}
\lim_{r\downarrow0}\mathfrak f_z(r)z(r)^{1-p}=0,
\qquad
\mathfrak f_z(r)z(r)^{1-p}
=o\bigl(r^{N-1}\mathfrak w(r)\bigr)
\qquad(r\downarrow0).
\label{eq:radial-center-condition}
\end{equation}
Define
\begin{equation*}
\theta_R(r):=
\frac{|z'(r)|^{p-1}}
{\mathfrak w(r)z(r)^{p-1}},
\qquad 0<r\le R.
\end{equation*}
We assume that $\theta_R$ is continuous on $(0,R]$ and nondecreasing.
\end{assumption}

The normalized radial flux has the prescribed endpoint values.
\begin{lemma}
\label{lem:abstract-flux-endpoints}
Under \Cref{ass:flux}, $\theta_R$ extends continuously to $[0,R]$ and
\[
\theta_R(0)=0,
\qquad
\theta_R(R)=\beta.
\]
In particular, $0\le\theta_R\le\beta$ on $[0,R]$.
\end{lemma}

\begin{proof}
By the definitions of $\mathfrak f_z$ and $\mathfrak w$,
\[
\theta_R(r)
=\frac{\mathfrak f_z(r)z(r)^{1-p}}
{r^{N-1}\mathfrak w(r)}.
\]
Thus \eqref{eq:radial-center-condition} gives $\theta_R(r)\to0$ as $r\downarrow0$. To identify the outer endpoint, take radial test functions in the weak eigenvalue equation and integrate the one-dimensional flux equation by parts. The interior terms cancel by \eqref{eq:radial-flux-equation}, leaving
\[
\mathfrak f_z(R)
=\beta R^{N-1}\mathfrak w(R)z(R)^{p-1}.
\]
Hence $\theta_R(R)=\beta$. The bounds follow from monotonicity.
\end{proof}

For $0<r\le R$, define
\begin{equation*}
\mathcal H_{B_R}^w(B_r,\theta_R)
:=
\frac1{\mu(B_r)}
\left[
\theta_R(r)P_w(B_r)
-
(p-1)\int_{B_r}\theta_R(|x|)^{p'}\,d\mu
\right].
\end{equation*}

The radial equation yields an exact identity for every concentric subball.
\begin{lemma}
Under \Cref{ass:flux}, for every $0<r\le R$,
\begin{equation*}
\mathcal H_{B_R}^w(B_r,\theta_R)
=
\lambda_{1,\beta}^{m,w}(B_R).
\end{equation*}
\end{lemma}

\begin{proof}
Write $\lambda_R:=\lambda_{1,\beta}^{m,w}(B_R)$. By \Cref{lem:abstract-flux-endpoints}, the endpoint values and bounds for $\theta_R$ are available. For every fixed $r>0$, strict decrease gives $z(s)\ge z(r)>0$ on $(0,r)$, and hence
\[
\int_0^r s^{N-1}\frac{|z'(s)|^p}{z(s)^p}\,ds
\le
z(r)^{-p}\int_0^r s^{N-1}|z'(s)|^p\,ds<\infty.
\]
Thus the product below is absolutely continuous up to the center. From \eqref{eq:radial-flux-equation},
\[
\frac{d}{ds}\left(\mathfrak f_z(s)z(s)^{1-p}\right)
=
\lambda_Rs^{N-1}\mathfrak m(s)
+
(p-1)s^{N-1}\frac{|z'(s)|^p}{z(s)^p}.
\]
Integrating from $0$ to $r$ and using the center limit, which also follows from \eqref{eq:radial-center-condition}, gives
\[
\lambda_R\int_0^r s^{N-1}\mathfrak m(s)\,ds
=
\mathfrak f_z(r)z(r)^{1-p}
-
(p-1)\int_0^r s^{N-1}\frac{|z'(s)|^p}{z(s)^p}\,ds.
\]
Multiplication by $\sigma_N$ and the identities
\[
\sigma_N\mathfrak f_z(r)z(r)^{1-p}
=
\theta_R(r)P_w(B_r),
\qquad
\theta_R(|x|)^{p'}\,d\mu
=
\frac{|\nabla z|^p}{z^p}\,dx
\]
prove the result.
\end{proof}

\subsection{A plateau-safe weighted rank rearrangement}

Define
\[
M(t):=\mu(\{\widetilde u>t\}),
\qquad
M_-(t):=\mu(\{\widetilde u\ge t\}),
\qquad 0<t\le1,
\]
and
\begin{equation*}
\mathcal A:=
\{t\in(0,1]:M_-(t)>M(t)\}
=
\{t\in(0,1]:\mu(\{\widetilde u=t\})>0\}.
\end{equation*}
The set $\mathcal A$ is at most countable and includes a possible maximum plateau $\{u=1\}$. Because $\mu=m\,dx$ is nonatomic, the standard isomorphism theorem for nonatomic standard probability spaces \cite[Vol.~II, Theorem~9.2.2]{Bogachev} gives, for every $t\in\mathcal A$, a Borel map
\[
q_t:\{\widetilde u=t\}\longrightarrow(M(t),M_-(t))
\]
which pushes $\mu\lfloor\{\widetilde u=t\}$ forward to Lebesgue measure on that interval.
Here and below $\widetilde u$ is chosen Borel measurable. Each plateau is therefore a Borel subset of the standard Borel space $\Omega$. After normalizing the two finite measures, the cited theorem gives a Borel isomorphism between conull Borel subsets of the plateau and the target interval. Extend it by one fixed target value on the Borel null complement. This produces a Borel map on the whole plateau without changing its pushforward.

Define the weighted rank function $Q:\Omega\to[0,\mu(\Omega)]$ by
\begin{equation}
Q(x):=
\begin{cases}
M(\widetilde u(x)),&\widetilde u(x)\notin\mathcal A,\\[1mm]
q_t(x),&\widetilde u(x)=t\in\mathcal A.
\end{cases}
\label{eq:rank-function}
\end{equation}
On the exceptional set where the precise representative is undefined, define $Q$ to be zero. That set is $\mathcal L^N$-null and therefore $\mu$-null because $\mu\ll\mathcal L^N$. Since $\mathcal A$ is countable, $M$ is monotone and hence Borel, and every $q_t$ is Borel on its Borel plateau, the countable piecewise definition makes $Q$ Borel.

The nonatomic rank construction transfers the comparison radius distribution to the superlevel geometry.
Here $Q_\#\mu$ denotes the pushforward of $\mu$ under $Q$,
whereas $\mathcal L^1\lfloor(0,\mu(\Omega))$ denotes
one-dimensional Lebesgue measure restricted to the interval
$(0,\mu(\Omega))$. Thus \eqref{eq:rank-pushforward} below means that, for
every Borel set $A\subset(0,\mu(\Omega))$,
\begin{equation*}
	\mu\bigl(Q^{-1}(A)\bigr)=\mathcal L^1(A).
\end{equation*}
Equivalently, $Q$ is measure preserving from $(\Omega,\mu)$ onto this
interval; after normalizing the two measures by $\mu(\Omega)$, the
variable $Q$ is uniformly distributed on $(0,\mu(\Omega))$. The
restriction notation used below has the analogous meaning:
$(\mu\lfloor U_t)(E)=\mu(E\cap U_t)$ for every Borel set $E$.
\begin{lemma}
\label{lem:rank-rearrangement}
The function $Q$ satisfies
\begin{equation}
Q_\#\mu
=
\mathcal L^1\lfloor(0,\mu(\Omega)).
\label{eq:rank-pushforward}
\end{equation}
If $t\in(0,1)\setminus\mathcal A$, then
\begin{equation*}
U_t=\{Q<M(t)\}
\end{equation*}
up to a $\mu$-null set, and the restriction of $Q$ to $U_t$ pushes $\mu\lfloor U_t$ forward to Lebesgue measure on $(0,M(t))$.

Set $\varrho(0):=0$, let
\begin{equation*}
r(t):=\varrho(M(t)),
\qquad
\rho(x):=\varrho(Q(x)),
\qquad
\varphi(x):=\theta_R(\rho(x)).
\end{equation*}
By \Cref{lem:abstract-flux-endpoints}, $0\le\varphi\le\beta$, so $\varphi\in L^{p'}(\Omega,d\mu)$. For almost every $t\in(0,1)\setminus\mathcal A$,
\begin{equation}
\varphi=\theta_R(r(t))
\qquad
\Haus^{N-1}\text{-a.e. on }\Sigma_t,
\label{eq:phi-on-level}
\end{equation}
and
\begin{equation}
\int_{U_t}\varphi^{p'}\,d\mu
=
\int_{B_{r(t)}}\theta_R(|x|)^{p'}\,d\mu.
\label{eq:equimeasurable-integral}
\end{equation}
\end{lemma}

\begin{proof}
The countability of $\mathcal A$ follows because its level sets are disjoint and have positive finite measure. Positivity of $u$ and the normalization $\|u\|_\infty=1$ give
\[
M(0+)=\mu(\Omega),\qquad M(1)=0.
\]
For $s\in(0,\mu(\Omega))$, define the generalized inverse
\[
t_s:=\sup\{t\in(0,1):M(t)>s\}.
\]
The one-sided limits of the decreasing distribution function are
$M(t_s)$ and $M_-(t_s)$, so
$M(t_s)\le s\le M_-(t_s)$. If $t_s\in\mathcal A$, then $\{Q<s\}$ consists, up to a null set, of $\{\widetilde u>t_s\}$ and the initial segment $\{x:\widetilde u(x)=t_s,\ q_{t_s}(x)<s\}$ of the plateau, whose measure is $s-M(t_s)$. If $t_s\notin\mathcal A$, the two endpoint values coincide and there is no plateau segment. Hence $\mu(\{Q<s\})=s$, proving \eqref{eq:rank-pushforward}.

If $t\notin\mathcal A$, no mass is assigned to the level $\{\widetilde u=t\}$. Applying the same distribution-function description at $s=M(t)$ gives $U_t=\{Q<M(t)\}$ modulo a null set. Restricting the preceding calculation to $0<s<M(t)$ proves the restricted pushforward assertion.

By \Cref{lem:radial-volume-coordinate}, the map $x\mapsto V(|x|)$ pushes $\mu\lfloor B_{r(t)}$ forward to Lebesgue measure on $(0,M(t))$. Therefore $\varrho(Q(x))$ on $U_t$ and $|x|$ on $B_{r(t)}$ have the same distribution. Composition with the nondecreasing function $\theta_R$ proves \eqref{eq:equimeasurable-integral}.

By the Federer--Vol'pert theorem for Sobolev functions \cite[Theorem~3.78]{AFP}, for almost every $t$ the precise representative satisfies
$\widetilde u=t$ at $\Haus^{N-1}$-almost every point of $\Sigma_t$. Moreover, if $N_u$ is the Lebesgue-null exceptional set on which the precise representative is not defined, the coarea formula gives
\[
\int_0^1
\Haus^{N-1}(N_u\cap\Sigma_t)\,dt
=
\int_{N_u}|\nabla u|\,dx
=0.
\]
Thus the arbitrary definition of $Q$ on $N_u$ does not affect almost any reduced level boundary. If $t\notin\mathcal A$, \eqref{eq:rank-function} gives $Q=M(t)$ on $\Sigma_t$ outside an $\Haus^{N-1}$-null set, and \eqref{eq:phi-on-level} follows.
\end{proof}

\subsection{Comparison and transfer}

At every admissible level, the constructed rank rearrangement produces the required comparison.
\begin{lemma}
\label{lem:H-comparison}
For almost every $t\in(0,1)$,
\begin{equation*}
\mathcal H_\Omega^w(U_t,\varphi)
\ge
\mathcal H_{B_R}^w(B_{r(t)},\theta_R)
=
\lambda_{1,\beta}^{m,w}(B_R).
\end{equation*}
\end{lemma}

\begin{proof}
Choose $t$ outside the null sets in the preceding lemmas and outside the countable set $\mathcal A$. By \eqref{eq:phi-on-level} and $0\le\theta_R(r(t))\le\beta$,
\begin{align*}
\int_{\Sigma_t}\varphi w\,d\Haus^{N-1}
+
\beta\int_{\Gamma_t}w\,d\Haus^{N-1}
&\ge
\theta_R(r(t))
\left[
\int_{\Sigma_t}w\,d\Haus^{N-1}
+
\int_{\Gamma_t}w\,d\Haus^{N-1}
\right]
\\
&=
\theta_R(r(t))P_w(U_t)
\\
&\ge
\theta_R(r(t))P_w(B_{r(t)}),
\end{align*}
where the equality follows from \Cref{lem:weighted-superlevel-perimeter} and the last inequality from \Cref{prop:superlevel-isop}. The negative terms agree by \eqref{eq:equimeasurable-integral}, and the denominators agree by definition of $r(t)$. The ball identity completes the proof.
\end{proof}

\begin{theorem}[Weighted Bossel--Daners transfer]
\label{thm:weighted-transfer}
Let $N\ge2$, $1<p<\infty$, and $\beta>0$. Let $m$ satisfy \Cref{ass:radial-volume-density}, choose a radial representative $m(x)=\mathfrak m(|x|)$ satisfying \eqref{eq:positive-radial-density}, and set $d\mu=m\,dx$ and $w=m^{1/p'}$. Let $\Omega\subset\R^N$ be a bounded connected Lipschitz domain with $0<\mu(\Omega)<\infty$, and let $\Omega^\sharp=B_R$. Assume:
\begin{enumerate}
\item $(m,w)$ is spectrally admissible on $\Omega$ and $B_R$;
\item \eqref{eq:boundary-weight-integrable} and \eqref{eq:ball-boundary-weight} hold;
\item \Cref{ass:double-density} holds;
\item \Cref{ass:flux} holds for the comparison-ball eigenfunction.
\end{enumerate}
Then
\begin{equation*}
\lambda_{1,\beta}^{m,w}(\Omega^\sharp)
\le
\lambda_{1,\beta}^{m,w}(\Omega).
\end{equation*}
\end{theorem}

\begin{proof}
Construct $\varphi$ by \Cref{lem:rank-rearrangement}. Since $0\le\varphi\le\beta$ and $\mu(\Omega)<\infty$, one has $\varphi\in L^{p'}(\Omega,d\mu)$. Let $\mathcal T$ be the intersection of all full-measure sets of admissible levels in the preceding lemmas, with $\mathcal A$ removed. By \Cref{lem:selection}, there exists $t\in\mathcal T$ such that
\[
\mathcal H_\Omega^w(U_t,\varphi)
\le
\lambda_{1,\beta}^{m,w}(\Omega).
\]
By \Cref{lem:H-comparison},
\[
\mathcal H_\Omega^w(U_t,\varphi)
\ge
\lambda_{1,\beta}^{m,w}(B_R).
\]
Since $B_R=\Omega^\sharp$, the conclusion follows.
\end{proof}

\begin{remark}
The comparison function automatically satisfies
$\int_\Omega\varphi^{p'}\,d\mu\le\beta^{p'}\mu(\Omega)$. Thus no auxiliary product-integrability condition involving
$(\varphi-\varphi_u)|\nabla u|w/u$ is required.
\end{remark}

\begin{remark}
The abstract theorem separates applications into three independent tasks:
compatible double-density isoperimetry, spectral admissibility, and verification of the regular monotone ball flux. The last item includes the radial weak-solution regularity and center behavior in \eqref{eq:radial-flux-def}--\eqref{eq:radial-center-condition}.
\end{remark}

\subsection{Fixed-volume consequence}

Let $m$ satisfy \Cref{ass:radial-volume-density}, choose a radial representative $m(x)=\mathfrak m(|x|)$ with
\[
0<\mathfrak m(r)<\infty\qquad(r>0),
\]
and set $w=m^{1/p'}$. For $v>0$, let $B_v^\sharp$ be the unique centered ball with $\mu(B_v^\sharp)=v$.

For the fixed-volume formulation, define the admissible class as follows.
\begin{definition}
\label{def:fixed-volume-class}
Fix $1<p<\infty$, $\beta>0$, and $v>0$. The class $\mathcal D_{m,w,\beta}(v)$ consists of bounded connected Lipschitz domains $\Omega$ such that:
\begin{enumerate}
\item $\mu(\Omega)=v$, so $\Omega^\sharp=B_v^\sharp$;
\item $w\in L^1(\partial\Omega)$ and $w\in L^1(\partial B_v^\sharp)$;
\item $(m,w)$ is spectrally admissible on $\Omega$ and $B_v^\sharp$;
\item the positive radial first eigenfunction on $B_v^\sharp$ satisfies the radial ball structure in \Cref{ass:flux}; its endpoint values then follow from \Cref{lem:abstract-flux-endpoints}.
\end{enumerate}
\end{definition}

The transfer theorem gives the corresponding fixed-volume minimization statement.
\begin{corollary}

Suppose that \Cref{ass:double-density} holds. Then, for every $\Omega\in\mathcal D_{m,w,\beta}(v)$,
\[
\lambda_{1,\beta}^{m,w}(B_v^\sharp)
\le
\lambda_{1,\beta}^{m,w}(\Omega).
\]
\end{corollary}

\begin{proof}
Every domain- and ball-specific assumption of \Cref{thm:weighted-transfer} is part of \Cref{def:fixed-volume-class}; the positivity of the radial density and the double-density inequality are the remaining global hypotheses. Hence the comparison follows from \Cref{thm:weighted-transfer}.
\end{proof}

\begin{remark}
The criterion in \Cref{thm:spectral-admissibility} is only one sufficient way to verify the spectral part. Other exponent or weight regimes may be covered by different spectral arguments.
\end{remark}

\begin{remark}
Connectedness is retained in the abstract weighted class because it enters positivity and simplicity, and a weighted component reduction need not preserve the comparison structure. The unweighted specialization admits a component reduction and therefore removes this restriction.
\end{remark}

\section{Consistency check: the unweighted Lipschitz case}

Here $m=w=1$. Thus $d\mu=dx$ and $P_w=P$. For a bounded Lipschitz open set, write
\[
\lambda_{1,\beta}(\Omega)
=
\inf_{0\ne u\in W^{1,p}(\Omega)}
\frac{\displaystyle\int_\Omega|\nabla u|^p\,dx
+\beta\int_{\partial\Omega}|\operatorname{Tr}u|^p\,d\Haus^{N-1}}
{\displaystyle\int_\Omega|u|^p\,dx}.
\]
This section is only a consistency check within the bounded Lipschitz class. The Robin Faber--Krahn inequality is known in substantially stronger finite-perimeter and free-discontinuity settings \cite{BucurGiacomini2010,BucurGiacomini2015}; see also the survey \cite{BucurFreitasKennedy}. The present transfer uses Lipschitz coercivity, the Sobolev/$BV$ boundary-trace formula, and a finite Lipschitz graph atlas for boundary measure estimates, so no extension to those broader classes is claimed here.

In the unweighted case, the abstract spectral assumptions hold automatically.
\begin{proposition}
\label{prop:unweighted-spectral}
Let $1<p<\infty$, $\beta>0$, and let $\Omega$ be a bounded connected Lipschitz domain. Then $(1,1)$ is spectrally admissible on $\Omega$.
\end{proposition}

\begin{proof}
For $1<p\le N$, the assertion is the special case $m=w=1$ of \Cref{thm:spectral-admissibility}. If $p>N$, compactness of the Sobolev and trace maps gives attainment and the weak equation, while the Robin form is coercive. Standard local regularity for the unweighted $p$-Laplace equation gives a continuous representative, and the weak Harnack argument used in \Cref{thm:first-eigenfunction-positivity} gives strict positivity. The Picone identity and its equality case \cite{AllegrettoHuang} give simplicity. These standard first-eigenfunction facts are also used in the unweighted Robin proof of Bucur and Daners \cite{BucurDaners}.
\end{proof}

\subsection{The radial ball flux}
Let $B_R$ be centered, let $\lambda_R=\lambda_{1,\beta}(B_R)$, and let $z$ be its positive first eigenfunction normalized by $\|z\|_{L^\infty(B_R)}=1$. For every orthogonal map $O$, $z\circ O$ is another normalized positive first eigenfunction. Simplicity therefore gives $z\circ O=z$, so $z(x)=Z(|x|)$ is radial.

The interior $C^{1,\alpha}$ regularity and rotational invariance give $Z'(0)=0$. Define
\[
\mathfrak f_R(r)=-r^{N-1}|Z'(r)|^{p-2}Z'(r).
\]
The radial weak equation therefore has no integration constant and gives
\[
\mathfrak f_R(r)=\lambda_R\int_0^r s^{N-1}Z(s)^{p-1}\,ds,
\qquad
\mathfrak f_R(R)=\beta R^{N-1}Z(R)^{p-1}.
\]
In particular, $\mathfrak f_R(r)>0$ for $r>0$, hence $Z'(r)<0$. Thus $Z$ is strictly decreasing, its maximum is attained at the center, and the normalization gives $Z(0)=1$. Moreover, $Z(R)>0$: otherwise the Robin identity would give $\mathfrak f_R(R)=0$, contradicting the positive integral formula. At the center,
\[
\mathfrak f_R(r)=\frac{\lambda_R}{N}r^N+o(r^N).
\]
Set
\[
\theta_R(r)=\frac{|Z'(r)|^{p-1}}{Z(r)^{p-1}}.
\]
Then $\theta_R\in C([0,R])\cap C^1((0,R])$,
\[
\theta_R(0)=0,\qquad \theta_R(R)=\beta,
\]
and
\begin{equation}
\theta_R'(r)
=
\lambda_R+(p-1)\theta_R(r)^{p'}-\frac{N-1}{r}\theta_R(r).
\label{eq:unweighted-flux-ode}
\end{equation}
Also $\theta_R(r)=\lambda_R r/N+o(r)$ near zero. Substitution in \eqref{eq:unweighted-flux-ode} gives
$\theta_R'(r)=\lambda_R/N+o(1)>0$ for sufficiently small $r$. If $r_0$ were the first zero of $\theta_R'$, differentiating \eqref{eq:unweighted-flux-ode} would give
\[
\theta_R''(r_0)=\frac{N-1}{r_0^2}\theta_R(r_0)>0,
\]
whereas a first zero approached from positive values requires $\theta_R''(r_0)\le0$. Hence:

The normalized unweighted radial flux is strictly increasing.
\begin{lemma}
\label{lem:unweighted-flux}
For every $R>0$,
\[
\theta_R'(r)>0\qquad(0<r<R),
\]
and therefore $0<\theta_R(r)<\beta$ on $(0,R)$.
\end{lemma}

\subsection{Connected domains}
By the Euclidean isoperimetric inequality, the radial identities above, and
\Cref{prop:unweighted-spectral} and \Cref{lem:unweighted-flux}, all hypotheses of
\Cref{thm:weighted-transfer} hold on every bounded connected Lipschitz domain.
Thus, if $|B_R|=|\Omega|$,
\[
\lambda_{1,\beta}(B_R)\le\lambda_{1,\beta}(\Omega).
\]

\subsection{Ball-radius monotonicity and component reduction}
The first Robin eigenvalue on a ball decreases strictly with its radius.
\begin{lemma}
\label{lem:unweighted-ball-radius}
If $0<r<R$, then
\[
\lambda_{1,\beta}(B_R)<\lambda_{1,\beta}(B_r).
\]
\end{lemma}

\begin{proof}
Let $q=R/r>1$ and let $z_r$ be a first eigenfunction on $B_r$. The trial function $v(x)=z_r(x/q)$ on $B_R$ gives
\[
\lambda_{1,\beta}(B_R)
\le
\frac{q^{-p}\int_{B_r}|\nabla z_r|^p
+\beta q^{-1}\int_{\partial B_r}|z_r|^p}
{\int_{B_r}|z_r|^p}
<
\lambda_{1,\beta}(B_r),
\]
because both scaling factors are strictly smaller than one and the original numerator is positive.
\end{proof}

By the standing convention from the introduction, the Sobolev space, boundary integral, and volume integral decompose componentwise on a bounded Lipschitz open set.

For a finite disjoint union, the first eigenvalue is the minimum over the components.
\begin{lemma}
If a bounded Lipschitz open set has connected components $\Omega_1,\dots,\Omega_k$, then
\[
\lambda_{1,\beta}(\Omega)=\min_{1\le j\le k}\lambda_{1,\beta}(\Omega_j).
\]
\end{lemma}

\begin{proof}
For $u\in W^{1,p}(\Omega)$, decompose its energy and mass over the components. Each component energy is at least $\lambda_{1,\beta}(\Omega_j)$ times the corresponding mass, so the quotient is bounded below by the minimum component eigenvalue. Conversely, extend a first eigenfunction of a minimizing component by zero to all other components.
\end{proof}

\begin{theorem}[Bucur--Daners inequality]
Let $\Omega\subset\R^N$ be a bounded Lipschitz open set satisfying the standing finite-component convention, and let $|B_R|=|\Omega|$. Then
\[
\lambda_{1,\beta}(B_R)\le\lambda_{1,\beta}(\Omega).
\]
\end{theorem}

\begin{proof}
Choose a component $\Omega_{j_0}$ with
\[
\lambda_{1,\beta}(\Omega)=\lambda_{1,\beta}(\Omega_{j_0}),
\]
and let $|B_{r_0}|=|\Omega_{j_0}|$. The connected comparison gives
\[
\lambda_{1,\beta}(B_{r_0})\le\lambda_{1,\beta}(\Omega_{j_0}).
\]
Since $r_0\le R$, \Cref{lem:unweighted-ball-radius} yields
\[
\lambda_{1,\beta}(B_R)
\le
\lambda_{1,\beta}(B_{r_0})
\le
\lambda_{1,\beta}(\Omega).
\]
\end{proof}

\section{Pure-power verification via known power-weight isoperimetry}

Throughout this section,
\[
N\ge2,\qquad 1<p\le N,\qquad -N<b<0.
\]
Set
\[
m_b(x)=|x|^b,\qquad w_b(x)=|x|^{b/p'},
\]
with arbitrary values at the origin, and define
\[
\mu_b(E)=\int_E|x|^b\,dx,
\qquad
P_b(E)=\int_{\partial^*E}|x|^{b/p'}\,d\Haus^{N-1}.
\]
For $0<\mu_b(E)<\infty$, let $E_b^\sharp$ be the centered ball of the same weighted volume.

For $R>0$,
\[
\mu_b(B_R)=\frac{\sigma_N}{N+b}R^{N+b},
\qquad
P_b(B_R)=\sigma_NR^{N-1+b/p'}.
\]
Hence the centered profile is continuous and strictly increasing.

\subsection{The radial transformation}
Set
\[
\gamma=\frac{N+b}{N}\in(0,1),
\qquad
T(0):=0,
\qquad
T(x)=|x|^{\gamma-1}x\quad(x\ne0).
\]
Then
\begin{equation}
|T(E)|=\gamma\mu_b(E).
\label{eq:power-volume-map}
\end{equation}
Define
\[
a=\frac{b(p-N)}{p(N+b)}.
\]
The assumptions imply $a\ge0$. For a finite-perimeter set $F$, put
\[
\mathcal P_a(F):=\int_{\partial^*F}|y|^a\,d\Haus^{N-1}(y).
\]

The radial transformation converts the compatible weighted perimeter into an increasing-power perimeter. Here a set is separated from the origin if $\inf_{x\in E}|x|>0$.
\begin{lemma}
\label{lem:power-perimeter-reduction}
If $E$ is bounded, has finite perimeter, and is separated from the origin, then
\[
P_b(E)\ge \mathcal P_a(T(E)).
\]
\end{lemma}

\begin{proof}
On a suitable annulus, $T$ and $T^{-1}$ are bi-Lipschitz. If $\nu_r=\nu_E(x)\cdot x/|x|$, the tangential Jacobian is
\[
J_{N-1}T(x,\nu_E)
=
|x|^{(N-1)(\gamma-1)}
\sqrt{\nu_r^2+\gamma^2(1-\nu_r^2)}
\le
|x|^{(N-1)(\gamma-1)}.
\]
The final inequality uses $0<\gamma<1$, equivalently $b<0$.
Since
\[
\gamma a+(N-1)(\gamma-1)=\frac b{p'},
\]
the area formula gives the result.
\end{proof}

\begin{theorem}[ABCMP power-weight theorem and its increasing-power specialization]
\label{thm:known-power-perimeter}
Let $k,\ell\in\R$, $\ell+N>0$. Theorem~1.1 of \cite{ABCMP} applies to smooth bounded sets if one of the following holds:
\begin{enumerate}
\item[$\mathrm{(i)}$] $N\ge1$ and $\ell+1\le k$;
\item[$\mathrm{(ii)}$] $N\ge2$, $k\le\ell+1$, and $\ell(N-1)/N\le k\le0$;
\item[$\mathrm{(iii)}$] $N\ge3$, $0\le k\le\ell+1$, and
\[
\frac1{\ell+N}
\ge
\frac1{k+N-1}
-
\frac{(N-1)^2}{N(k+N-1)^3};
\]
\item[$\mathrm{(iv)}$] $N=2$, $k\le\ell+1$, and either
\[
\ell\le0\le k\le\frac13,
\]
or
\[
k\ge\frac13
\qquad\text{and}\qquad
\frac1{\ell+2}
\ge
\frac1{k+1}
-
\frac{16}{27(k+1)^3}.
\]
\end{enumerate}
Under these hypotheses it gives
\[
\int_{\partial F}|x|^k\,d\Haus^{N-1}
\ge
C_{k,\ell,N}^{\mathrm{rad}}
\left(\int_F|x|^\ell\,dx\right)^{(k+N-1)/(\ell+N)},
\]
where
\[
C_{k,\ell,N}^{\mathrm{rad}}
:=
\frac{\displaystyle\int_{\partial B_1}|x|^k\,d\Haus^{N-1}}
{\displaystyle\left(\int_{B_1}|x|^\ell\,dx\right)^{(k+N-1)/(\ell+N)}}.
\]
Equality holds for centered balls. In particular, setting $\ell=0$ shows that, for every $a\ge0$, every bounded finite-perimeter set $F\subset\R^N$ with $0<|F|<\infty$ satisfies
\[
\mathcal P_a(F)
\ge
\sigma_N\left(\frac{|F|}{\omega_N}\right)^{(N-1+a)/N}.
\]
Equivalently, centered balls minimize $\mathcal P_a$ under fixed Lebesgue volume.
\end{theorem}

\begin{proof}
The first assertion is precisely \cite[Theorem~1.1]{ABCMP}. When $\ell=0$, cases $\mathrm{(ii)}$--$\mathrm{(iv)}$ cover $0\le a\le1$, as noted explicitly after that theorem, and case $\mathrm{(i)}$ covers $a\ge1$.

It remains to justify the finite-perimeter extension used here. If $\mathcal P_a(F)=\infty$, the conclusion is immediate. Otherwise, choose a ball $B_{R_*}$ containing $F$. Because $g(x):=|x|^a$ is continuous and bounded on that ball, strict $BV$ approximation with support in a fixed slightly larger ball, together with Reshetnyak continuity \cite{AFP,Maggi}, provides $u_j\in C_c^\infty(\R^N)$, $0\le u_j\le1$, such that
\[
u_j\to\chi_F\quad\text{in }L^1,
\qquad
\int_{\R^N}g|\nabla u_j|\,dx\to\mathcal P_a(F).
\]
Choose $\eta_j\downarrow0$ so that
$\|u_j-\chi_F\|_{L^1}/\eta_j\to0$. By coarea and Sard's theorem there is a regular value $t_j\in(\eta_j,1-\eta_j)$ such that, for $F_j:=\{u_j>t_j\}$,
\[
\mathcal P_a(F_j)
\le\frac1{1-2\eta_j}
\int_{\R^N}g|\nabla u_j|\,dx.
\]
Moreover,
\[
|F_j\triangle F|
\le\frac{\|u_j-\chi_F\|_{L^1}}
{\min\{t_j,1-t_j\}}
\longrightarrow0.
\]
Lower semicontinuity of weighted total variation gives
$\mathcal P_a(F)\le\liminf_j\mathcal P_a(F_j)$, while the preceding bound gives the reverse limsup. Thus $|F_j|\to|F|$ and $\mathcal P_a(F_j)\to\mathcal P_a(F)$. Applying the smooth inequality to $F_j$ and passing to the limit proves the displayed finite-perimeter inequality.
\end{proof}

\begin{theorem}[Pure-power double-density inequality]
\label{thm:power-double-density}
Every bounded finite-perimeter set $E$ with $0<\mu_b(E)<\infty$ satisfies
\[
P_b(E)\ge P_b(E_b^\sharp).
\]
\end{theorem}

\begin{proof}
If $P_b(E)=+\infty$, the conclusion is immediate. We therefore assume $P_b(E)<\infty$.

Assume first that $E$ is separated from the origin and set $F=T(E)$. By \Cref{lem:power-perimeter-reduction} and \Cref{thm:known-power-perimeter},
\[
P_b(E)\ge \mathcal P_a(F)\ge \mathcal P_a(B_\rho),
\]
where $|B_\rho|=|F|$. If $E_b^\sharp=B_R$, then \eqref{eq:power-volume-map} gives $\rho=R^\gamma$, and
\[
\gamma(N-1+a)=N-1+\frac b{p'}.
\]
Therefore $\mathcal P_a(B_\rho)=P_b(B_R)$.

For a general $E$, take $E_\varepsilon=E\setminus\overline{B_\varepsilon}$ at radii for which the truncation has finite perimeter. Then
\[
P_b(E_\varepsilon)
\le
\int_{\partial^*E\cap\{|x|>\varepsilon\}}|x|^{b/p'}\,d\Haus^{N-1}
+
\sigma_N\varepsilon^{N-1+b/p'}.
\]
Here
\[
N-1+\frac b{p'}
>
N-1-\frac N{p'}
=\frac Np-1
\ge0,
\]
where the strict inequality uses $b>-N$ and the last one uses $p\le N$. Hence, along a sequence $\varepsilon_j\downarrow0$,
\[
\limsup_jP_b(E_{\varepsilon_j})\le P_b(E),
\qquad
\mu_b(E_{\varepsilon_j})\to\mu_b(E).
\]
Apply the separated-set result to $E_{\varepsilon_j}$ and pass to the limit using continuity of the centered profile.
\end{proof}

\begin{remark}
The theorem is a direct specialization of known power-weight isoperimetry and is not claimed as a new geometric result. At $p=N$, one has $a=0$ and the transformed inequality becomes the ordinary Euclidean isoperimetric inequality; the spectral theory in \Cref{sec:spectral} covers this endpoint as well.
\end{remark}

\section{Relation with the weighted Robin results of Amato--Chiacchio--Gentile}

Amato--Chiacchio--Gentile \cite{ACG} study weighted $p$-Poisson equations with variable Robin boundary coefficients. Their standing hypotheses $\mathrm{(H_1)}$--$\mathrm{(H_3)}$ are
\[
p\ge N,
\qquad
-N<\ell<0,
\qquad
0<\inf_{\partial\Omega}\beta(x)\le
\sup_{\partial\Omega}\beta(x)<\infty,
\]
where $\Omega$ is bounded and Lipschitz and $0\notin\partial\Omega$. Their problem is
\begin{equation*}
\begin{cases}
-\Delta_pu=f(x)|x|^\ell&\text{in }\Omega,\\
|\nabla u|^{p-2}\partial_\nu u+\beta(x)|u|^{p-2}u=0&\text{on }\partial\Omega.
\end{cases}
\end{equation*}
Their weighted volume and perimeter are
\[
|E|_\ell=\int_E|x|^\ell\,dx,
\qquad
P_{\ell/p'}(E)=\int_{\partial^*E}|x|^{\ell/p'}\,d\Haus^{N-1}.
\]
Thus, after setting $b=\ell$, the underlying double-density pair is the same formal pair
\[
m_b(x)=|x|^b,\qquad w_b(x)=|x|^{b/p'}.
\]
They define
\[
\widetilde\beta
:=\inf_{x\in\partial\Omega}
\beta(x)|x|^{-\ell/p'}.
\]
Their Theorem~1.3 states
\[
\lambda_{1,\beta(x)}(\Omega)
\ge
\lambda_{1,\widetilde\beta}(\Omega^\sharp),
\]
where the comparison ball carries the boundary coefficient
$\widetilde\beta R^{\ell/p'}$. For the special choice
$\beta(x)=\beta|x|^{\ell/p'}$, one has $\widetilde\beta=\beta$, so both Rayleigh quotients and the comparison conclusion agree exactly with the compatible quotient considered here. Their geometric input is the same double-density inequality, used as \cite[Theorem~2.1]{ACG} and traced there to the power-weight theory including \cite{ABCMP}.

The objectives and proof mechanisms are different. Amato--Chiacchio--Gentile derive weighted Talenti-type solution comparisons and then the eigenvalue consequence, and their framework allows a variable Robin coefficient. The present paper instead isolates a Bossel--Daners level-set transfer based on the exact $BV$ superlevel perimeter formula, a plateau-safe rank rearrangement, and a radial flux analysis on the comparison ball.

The exponent regimes are also different. Our pure-power theorem is established independently for
\[
1<p\le N,\qquad -p<b<0,
\]
by combining the geometric and spectral inputs established above with a direct ball-flux analysis.

For $1<p<N$, this is complementary to the hypothesis $p\ge N$ in \cite{ACG}. At the only shared exponent, $p=N$, and under the intersecting domain hypotheses, the two theorems give the same inequality for $\beta(x)=\beta|x|^{b/p'}$. The additional conclusion here at that endpoint is that the same inequality remains valid when $0\in\partial\Omega$, a contact excluded in \cite{ACG}; the proof is also independent of their Talenti comparison.

\begin{remark}
The double-density inequality is a known geometric input shared by the two approaches. The contribution here is the abstract weighted Bossel--Daners transfer and its combination with the direct proof of strict pure-power radial flux monotonicity.
\end{remark}

\section{The pure-power Robin Faber--Krahn inequality}

Throughout this section,
\begin{equation*}
N\ge2,\qquad 1<p\le N,\qquad -p<b<0.
\end{equation*}
Set
\begin{equation*}
m_b(x):=|x|^b,
\qquad
w_b(x):=|x|^{b/p'}
\qquad(x\ne0),
\end{equation*}
with arbitrary finite values at the origin when the weights are viewed as Lebesgue representatives.

For every bounded Lipschitz open set $\Omega\subset\R^N$, define
\begin{equation*}
\lambda_{1,\beta}^{b}(\Omega)
:=
\inf_{0\ne u\in W^{1,p}(\Omega)}
\frac{
\displaystyle\int_\Omega|\nabla u|^p\,dx
+
\beta\int_{\partial\Omega}|x|^{b/p'}|\operatorname{Tr}u|^p\,d\Haus^{N-1}}
{\displaystyle\int_\Omega|x|^b|u|^p\,dx}.
\end{equation*}

\subsection{Radial weak-solution structure}

Let $B_R$ be a centered ball and write
\[
\lambda_R:=\lambda_{1,\beta}^{b}(B_R).
\]
By \Cref{cor:power-spectral}, the first eigenvalue is attained by a positive bounded eigenfunction and is simple. Since the domain and both densities are rotationally invariant, \Cref{cor:radiality-ball} implies that the positive first eigenfunction is radial. Let
\[
z(x)=Z(|x|)
\]
be the radial representative, normalized by
\begin{equation*}
\|z\|_{L^\infty(B_R)}=1.
\end{equation*}
It satisfies
\begin{align}
&\int_{B_R}
|\nabla z|^{p-2}\nabla z\cdot\nabla\zeta\,dx
+
\beta\int_{\partial B_R}
|x|^{b/p'}z^{p-1}\zeta\,d\Haus^{N-1}
=
\lambda_R\int_{B_R}|x|^bz^{p-1}\zeta\,dx
\label{eq:power-ball-problem}
\end{align}
for every $\zeta\in W^{1,p}(B_R)$.

The weak radial equation can be written in integrated flux form.
\begin{lemma}
\label{lem:power-radial-regularity}
The radial representative $Z$ is locally absolutely continuous on $(0,R]$. The distributional radial flux
\begin{equation*}
\mathfrak f(r):=-r^{N-1}|Z'(r)|^{p-2}Z'(r)
\end{equation*}
admits an absolutely continuous representative on $[0,R]$ and satisfies
\begin{equation}
\mathfrak f(r)
=
\lambda_R\int_0^r
s^{N-1+b}Z(s)^{p-1}\,ds
\qquad(0\le r\le R).
\label{eq:power-radial-integral}
\end{equation}
Consequently,
\begin{equation*}
\mathfrak f(0)=0,
\qquad
\mathfrak f'(r)=\lambda_Rr^{N-1+b}Z(r)^{p-1}
\quad\text{for a.e. }r\in(0,R).
\end{equation*}
Moreover,
\begin{equation*}
Z'(r)<0
\qquad(0<r\le R),
\end{equation*}
$Z$ is strictly decreasing on $(0,R]$,
\begin{equation}
\lim_{r\downarrow0}Z(r)=1,
\label{eq:power-center-limit}
\end{equation}
and
\begin{equation}
\mathfrak f(R)
=
\beta R^{N-1+b/p'}Z(R)^{p-1},
\qquad
Z(R)>0.
\label{eq:power-boundary-flux}
\end{equation}
\end{lemma}

\begin{proof}
Since $z\in W^{1,p}(B_R)$ is radial, $Z$ is absolutely continuous on every compact subinterval of $(0,R]$. Taking radial test functions compactly supported in $(0,R)$ in \eqref{eq:power-ball-problem} gives, in distributions,
\[
\mathfrak f'(r)=\lambda_Rr^{N-1+b}Z(r)^{p-1}.
\]
The right-hand side belongs to $L^1(0,R)$ because $b>-N$ and $Z\in L^\infty(0,R)$. Thus the distributional flux admits an absolutely continuous representative with a finite one-sided limit at zero, and
\[
\mathfrak f(r)=C+
\lambda_R\int_0^r s^{N-1+b}Z(s)^{p-1}\,ds
\]
for some constant $C$.

If $C\ne0$, then $|\mathfrak f(r)|\ge|C|/2$ near zero, and
\[
|Z'(r)|^pr^{N-1}
=
|\mathfrak f(r)|^{p'}r^{-(N-1)/(p-1)}
\ge
c r^{-(N-1)/(p-1)}.
\]
Since $p\le N$, one has $(N-1)/(p-1)\ge1$, so the right-hand side is not integrable near zero. This contradicts $z\in W^{1,p}(B_R)$. Hence $C=0$, proving \eqref{eq:power-radial-integral}. This is the precise point at which $p\le N$ enters the radial integration.

Since $Z$ is continuous on compact subintervals of $(0,R]$, \eqref{eq:power-radial-integral} gives $\mathfrak f\in C^1((0,R])$. The identity
\[
Z'(r)
=-
\left(\frac{\mathfrak f(r)}{r^{N-1}}\right)^{1/(p-1)}
\quad\text{for a.e. }r\in(0,R)
\]
therefore provides a continuous representative of $Z'$ on $(0,R]$ and a $C^1$ representative of $Z$ there. The right-hand side of \eqref{eq:power-radial-integral} is strictly positive for $r>0$, so, for this representative,
\[
Z'(r)
=-\left(\frac{\mathfrak f(r)}{r^{N-1}}\right)^{1/(p-1)}<0
\qquad(0<r\le R).
\]
For $0<r_1<r_2\le R$,
\[
Z(r_2)-Z(r_1)
=-\int_{r_1}^{r_2}
\left(\frac{\mathfrak f(s)}{s^{N-1}}\right)^{1/(p-1)}ds<0,
\]
which proves strict decrease. Since $Z$ is positive, bounded, and decreasing, its limit at zero exists and equals its essential supremum; the normalization gives \eqref{eq:power-center-limit}.

To identify the boundary condition, use radial test functions in \eqref{eq:power-ball-problem} which are constant near the origin. Integration by parts, together with $\mathfrak f(0)=0$, yields
\[
\mathfrak f(R)=\beta R^{N-1+b/p'}Z(R)^{p-1}.
\]
If $Z(R)=0$, then $\mathfrak f(R)=0$, contradicting \eqref{eq:power-radial-integral}. Hence $Z(R)>0$.
\end{proof}

The integrated equation determines the center behavior of $Z$ and its flux.
\begin{lemma}
\label{lem:power-center-asymptotics}
As $r\downarrow0$,
\begin{equation}
\mathfrak f(r)
=
\frac{\lambda_R}{N+b}r^{N+b}
+o(r^{N+b}),
\label{eq:power-flux-asymptotic}
\end{equation}
and
\begin{equation}
|Z'(r)|^{p-1}
=
\frac{\lambda_R}{N+b}r^{1+b}
+o(r^{1+b}).
\label{eq:power-derivative-asymptotic}
\end{equation}
Consequently,
\begin{equation}
\mathfrak f(r)Z(r)^{1-p}
=o\bigl(r^{N-1+b/p'}\bigr)
\qquad(r\downarrow0).
\label{eq:power-center-flux-condition}
\end{equation}
\end{lemma}

\begin{proof}
By \eqref{eq:power-radial-integral} and \eqref{eq:power-center-limit},
\[
\frac{\mathfrak f(r)}{r^{N+b}}
=
\lambda_R\int_0^1
\tau^{N-1+b}Z(r\tau)^{p-1}\,d\tau
\longrightarrow
\frac{\lambda_R}{N+b}
\]
by dominated convergence. Since
$|Z'(r)|^{p-1}=r^{1-N}\mathfrak f(r)$, \eqref{eq:power-derivative-asymptotic} follows. Finally, $Z(r)\to1$ and
\[
(N+b)-(N-1+b/p')=1+\frac bp>0,
\]
so \eqref{eq:power-center-flux-condition} follows from \eqref{eq:power-flux-asymptotic}.
\end{proof}

\subsection{Automatic monotonicity of the radial flux}

Define
\begin{equation*}
\theta_R(r):=
\frac{|Z'(r)|^{p-1}}
{r^{b/p'}Z(r)^{p-1}}
=
\frac{\mathfrak f(r)}
{r^{N-1+b/p'}Z(r)^{p-1}},
\qquad 0<r\le R.
\end{equation*}

The normalized pure-power radial flux is strictly increasing.
\begin{proposition}
\label{prop:power-flux}
One has
\[
\theta_R\in C([0,R])\cap C^1((0,R]),
\]
with
\begin{equation*}
\theta_R(0)=0,
\qquad
\theta_R(R)=\beta,
\end{equation*}
and
\begin{equation}
\theta_R'(r)>0
\qquad(0<r<R).
\label{eq:power-flux-positive}
\end{equation}
Consequently,
\begin{equation}
0<\theta_R(r)<\beta
\qquad(0<r<R).
\label{eq:power-flux-bounds}
\end{equation}
Thus the pure-power ball eigenfunction verifies the radial ball structure in \Cref{ass:flux}.
\end{proposition}

\begin{proof}
By \Cref{lem:power-radial-regularity}, $\mathfrak f\in C^1((0,R])$, and the identity for $Z'$ gives $Z\in C^1((0,R])$. Thus $\theta_R\in C^1((0,R])$. Differentiation gives
\[
\theta_R'(r)
=
\lambda_Rr^{b/p}
-
\left(N-1+\frac b{p'}\right)\frac{\theta_R(r)}r
+
(p-1)\theta_R(r)\frac{-Z'(r)}{Z(r)}.
\]
Since
\[
\left(\frac{-Z'(r)}{Z(r)}\right)^{p-1}
=r^{b/p'}\theta_R(r),
\]
one obtains
\begin{equation}
\theta_R'(r)
=
r^{b/p}
\left[\lambda_R+(p-1)\theta_R(r)^{p'}\right]
-
\left(N-1+\frac b{p'}\right)\frac{\theta_R(r)}r.
\label{eq:power-flux-ode}
\end{equation}
Set
\[
\delta:=1+\frac bp>0,
\qquad
A:=N-1+\frac b{p'}>N-p\ge0,
\]
and introduce
\[
t:=\frac{r^\delta}{\delta},
\qquad
\Theta(t):=\theta_R(r),
\qquad
c:=\frac A\delta>0,
\qquad
t_R:=\frac{R^\delta}{\delta}.
\]
Then \eqref{eq:power-flux-ode} becomes, for $0<t\le t_R$,
\begin{equation}
\Theta'(t)
=
\lambda_R+(p-1)\Theta(t)^{p'}-\frac ct\Theta(t).
\label{eq:transformed-flux}
\end{equation}
By \Cref{lem:power-center-asymptotics},
\[
\theta_R(r)
=
\frac{\lambda_R}{N+b}r^\delta+o(r^\delta).
\]
Since $r^\delta=\delta t$ and $c+1=(N+b)/\delta$,
\[
\Theta(t)=\frac{\lambda_R}{c+1}t+o(t).
\]
Substituting this expansion into \eqref{eq:transformed-flux} gives
\[
\lim_{t\downarrow0}\Theta'(t)
=
\lambda_R-\frac{c\lambda_R}{c+1}
=
\frac{\lambda_R}{c+1}>0.
\]
Hence $\Theta'>0$ near zero.

Suppose that $t_0\in(0,t_R)$ is the first zero of $\Theta'$. Since the right-hand side of \eqref{eq:transformed-flux} is $C^1$ for $t>0$ and $\Theta(t)>0$, one has $\Theta\in C^2$ away from zero. Differentiating gives
\[
\Theta''(t)
=
\left[(p-1)p'\Theta(t)^{p'-1}-\frac ct\right]\Theta'(t)
+
\frac c{t^2}\Theta(t).
\]
Thus
\[
\Theta''(t_0)=\frac c{t_0^2}\Theta(t_0)>0.
\]
But $\Theta'>0$ immediately to the left of its first zero, so differentiability gives $\Theta''(t_0)\le0$, a contradiction. Therefore $\Theta'(t)>0$ for $0<t<t_R$, and hence \eqref{eq:power-flux-positive} holds.

The center value follows from the asymptotic formula. The boundary identity in \eqref{eq:power-boundary-flux} gives
\[
\theta_R(R)
=
\frac{\mathfrak f(R)}
{R^{N-1+b/p'}Z(R)^{p-1}}
=
\beta.
\]
Strict monotonicity gives \eqref{eq:power-flux-bounds}. The regularity and center conditions in \Cref{ass:flux} follow from \Cref{lem:power-radial-regularity,lem:power-center-asymptotics}.
\end{proof}

\subsection{The pure-power comparison and component reduction}

For a bounded Lipschitz open set $\Omega\subset\R^N$, let $\Omega_b^\sharp$ be the centered ball satisfying
\begin{equation*}
\int_{\Omega_b^\sharp}|x|^b\,dx
=
\int_\Omega|x|^b\,dx.
\end{equation*}

The compatible pure-power eigenvalue on a ball decreases strictly with its radius.
\begin{lemma}
\label{lem:power-ball-radius}
If $0<r<R$, then
\[
\lambda_{1,\beta}^{b}(B_R)
<
\lambda_{1,\beta}^{b}(B_r).
\]
\end{lemma}

\begin{proof}
Set $q:=R/r>1$, let $z_r$ be a first eigenfunction on $B_r$, and define $v(x):=z_r(x/q)$ on $B_R$. A change of variables gives
\begin{align*}
\lambda_{1,\beta}^{b}(B_R)
&\le
\frac{
q^{-p-b}\displaystyle\int_{B_r}|\nabla z_r|^p\,dx
+
\beta q^{-1-b/p}\displaystyle\int_{\partial B_r}|x|^{b/p'}|z_r|^p\,d\Haus^{N-1}}
{\displaystyle\int_{B_r}|x|^b|z_r|^p\,dx}
\\
&<
\lambda_{1,\beta}^{b}(B_r).
\end{align*}
Indeed, $p+b>0$ and $1+b/p>0$, so both scaling factors in the numerator are strictly smaller than one, while the unscaled numerator is positive.
\end{proof}

The pure-power quotient also decomposes componentwise.
\begin{lemma}
\label{lem:power-component-formula}
If $\Omega$ has connected components $\Omega_1,\dots,\Omega_k$, then
\[
\lambda_{1,\beta}^{b}(\Omega)
=
\min_{1\le j\le k}\lambda_{1,\beta}^{b}(\Omega_j).
\]
\end{lemma}

\begin{proof}
The weighted mass, gradient energy, and weighted boundary term decompose over the components. Hence every Rayleigh quotient is bounded below by the smallest component eigenvalue. Conversely, a first eigenfunction on a minimizing component, extended by zero to the other components, is an admissible competitor on $\Omega$.
\end{proof}

\begin{theorem}[Pure-power weighted Robin Faber--Krahn inequality]
\label{thm:power-robin-fk}
Let
\[
N\ge2,
\qquad
1<p\le N,
\qquad
-p<b<0,
\qquad
\beta>0.
\]
Then every bounded Lipschitz open set $\Omega\subset\R^N$ satisfying the standing finite-component convention satisfies
\begin{equation*}
\lambda_{1,\beta}^{b}(\Omega_b^\sharp)
\le
\lambda_{1,\beta}^{b}(\Omega).
\end{equation*}
Thus centered balls minimize the first compatible pure-power Robin eigenvalue under prescribed $|x|^b$-weighted volume.
\end{theorem}

\begin{proof}
Assume first that $\Omega$ is connected. By \Cref{cor:power-spectral}, the pair $(m_b,w_b)$ is spectrally admissible on $\Omega$ and $\Omega_b^\sharp$, and $w_b\in L^1$ on both boundaries. By \Cref{thm:power-double-density}, the compatible double-density inequality holds for every bounded finite-perimeter set of positive finite $\mu_b$-volume. By \Cref{lem:power-radial-regularity,lem:power-center-asymptotics} and \Cref{prop:power-flux}, the comparison-ball eigenfunction verifies every part of \Cref{ass:flux}. All hypotheses of \Cref{thm:weighted-transfer} are therefore satisfied, proving the comparison in the connected case.

For a general $\Omega$, choose a component $\Omega_{j_0}$ such that, by \Cref{lem:power-component-formula},
\[
\lambda_{1,\beta}^{b}(\Omega)
=
\lambda_{1,\beta}^{b}(\Omega_{j_0}).
\]
Let $(\Omega_{j_0})_b^\sharp=B_{r_0}$ and $\Omega_b^\sharp=B_R$. Since
$\mu_b(\Omega_{j_0})\le\mu_b(\Omega)$, one has $r_0\le R$. The connected comparison and \Cref{lem:power-ball-radius} yield
\[
\lambda_{1,\beta}^{b}(\Omega_b^\sharp)
=
\lambda_{1,\beta}^{b}(B_R)
\le
\lambda_{1,\beta}^{b}(B_{r_0})
\le
\lambda_{1,\beta}^{b}(\Omega_{j_0})
=
\lambda_{1,\beta}^{b}(\Omega).
\]
\end{proof}

The main theorem has the following fixed-volume formulation.
\begin{corollary}
Let $v>0$, and let $B_R$ be the unique centered ball satisfying
\[
\int_{B_R}|x|^b\,dx=v.
\]
Then
\[
\lambda_{1,\beta}^{b}(B_R)
=
\min\left\{
\lambda_{1,\beta}^{b}(\Omega):
\begin{array}{l}
\Omega\subset\R^N\text{ a bounded Lipschitz open set satisfying}\\
\text{the standing finite-component convention},\\[1mm]
\displaystyle\int_\Omega|x|^b\,dx=v
\end{array}
\right\}.
\]
\end{corollary}

\begin{proof}
Every admissible domain of weighted volume $v$ has $B_R$ as its centered comparison ball. Apply \Cref{thm:power-robin-fk}.
\end{proof}

\begin{remark}
The double-density theorem is a known geometric input. The new pure-power conclusion comes from combining it with the abstract weighted transfer and the automatic flux monotonicity in \Cref{prop:power-flux}.
\end{remark}

\begin{remark}
The strict ball-radius monotonicity shows that equality for a disconnected open set is impossible: a minimizing component has strictly smaller weighted volume than the whole set. Thus equality forces connectedness. A rigidity characterization within the connected class is not pursued here.
\end{remark}

\begin{remark}
As proved in \Cref{prop:hardy-threshold}, when $p<N$ and $0\in\Omega$, $b=-p$ is the Hardy critical scale: for $-N<b<-p$ the quotient collapses to zero, while at $b=-p$ the mass embedding loses compactness. When $p=N$, the endpoint weight $|x|^{-N}$ is not locally integrable. The geometric inequality itself remains valid in the larger range $-N<b<0$. The range $p>N$ lies outside the transformation used here because then $a=b(p-N)/(p(N+b))<0$, so \Cref{thm:known-power-perimeter} no longer supplies the required centered-ball inequality. No assertion is made in that regime.
\end{remark}

\end{document}